\documentclass[12pt, reqno]{amsart}

\usepackage[T1]{fontenc}
\usepackage[utf8]{inputenc}
\usepackage[english]{babel}

\usepackage{xcolor}

\definecolor{bleu_sombre}{rgb}{0,0,0.6}  \definecolor{rouge_sombre}{rgb}{0.8,0,0}\definecolor{vert_sombre}{rgb}{0,0.6,0}
\usepackage[plainpages=false,colorlinks,linkcolor=bleu_sombre,
citecolor=rouge_sombre,urlcolor=vert_sombre,breaklinks]{hyperref}
\usepackage{amsmath,amssymb,amsthm,graphicx,amsfonts,enumerate,stmaryrd, mathrsfs, tikz, dsfont}

\theoremstyle{plain}
\newtheorem{theorem}{{Theorem}}[section]
\newtheorem*{theorem*}{{Theorem}}
\newtheorem{proposition}[theorem]{Proposition}

\newtheorem*{proposition*}{Proposition}
\newtheorem{corollary}[theorem]{Corollary}
\newtheorem*{corollary*}{Corollary}
\newtheorem{lemma}[theorem]{Lemma}
\newtheorem{assumption}[theorem]{Assumption}
\newtheorem*{lemma*}{Lemma}

\theoremstyle{definition}

\newtheorem*{definition*}{Definition}

\theoremstyle{remark}
\newtheorem{remark}[theorem]{Remark}

\makeatletter

\@addtoreset{equation}{section}
\makeatother

\renewcommand{\leq}{\leqslant}	\renewcommand{\geq}{\geqslant}

\newcommand{\R}{\mathbb{R}}	

\newcommand{\N}{\mathbb{N}}

\renewcommand{\Re}{\mathrm{Re}\,}

\title[]{An example of accurate microlocal tunneling\\ in one dimension}

\author[A. Duraffour]{Antide Duraffour}
\address[A. Duraffour]{Univ. Rennes, IRMAR, Campus de Beaulieu,
bâtiments 22 et 23,
263 avenue du Général Leclerc, 35042 Rennes, France}
\email{antide.duraffour@univ-rennes.fr}

\author[N. Raymond]{Nicolas Raymond}
\address[N. Raymond]{Univ. Angers, CNRS, LAREMA, 2 Boulevard de Lavoisier, 49000 Angers}
\email{nicolas.raymond@univ-angers.fr}

\begin{document}
		\begin{abstract}
We investigate the spectral analysis of a class of pseudo-differential operators in one dimension. Under symmetry assumptions, we prove an asymptotic formula for the splitting of the first two eigenvalues. This article is an example of extension to pseudo-differential operators of the tunneling effect formulas known for the symmetric electric Schrödinger operator.
	\end{abstract}
	\maketitle
	
%\tableofcontents
	
\section{Introduction}

\subsection{Motivation}

This article is devoted to the spectral analysis of a microlocal version of the Schrödinger operator $-h^2\Delta+V(x)$.  This differential operator has drawn a lot of attention since the eighties, especially with the mathematical study of quantum tunnelling. A manifestation of this phenomenon is the effect of symmetries of $V$ on the spectrum of the operator in the semiclassical limit $h\to 0$ and has generated much interest lately. The most prominent results in this direction go back to the papers by Simon \cite{S1, S2, S3, S4} and the famous Helffer-Sjöstrand series of articles \cite{HS84, HS85c, HS85b, HS85d} establishing tunneling formulas. These articles motivated the study of purely magnetic tunneling effects, see for instance \cite{BHR16, BHR22}. These recent works have cast a new light on the semiclassical analysis of the magnetic Schrödinger operator by revealing the central role of the microlocal approach, developped for instance in \cite{Sj1982} and also in \cite{MT, martinez_andre_introduction_2002,  MS, Naktun, Shu1998}, to tackle spectral problems (see  \cite{AA23} where this view point has recently been used). Especially, the core of the strategy is a microlocal dimensional reduction that leads to an effective pseudo-differential operator in one dimension, even though the original operator is differential. These recent advances lead us to explore tunneling effect for pseudo-differential operators. In this context, already available results deal with operators of the form $a(h D_x) + V(x)$ such as the Klein-Gordon operator of symbol $\sqrt{1+\xi^2} + V(x)$ (see \cite[Section 3.2.3]{KG1994}) as well as the Harper operator (with symbol $\cos(\xi) + \cos(x)$) on $L^2(\R)$ (see \cite{HI, HII, HIII}). In the case when the $x$ and $\xi$ dependency are intertwined, there is no general result as accurate as that established in the Helffer-Sjöstrand papers (or in the works on the Witten Laplacian, see, for instance, \cite{HKN04,Nier}). The reason for that is the absence of exponentially sharp estimates for the eigenfunctions in general, even though there exist a priori bounds, see for instance \cite{MS}.

In the present article, we tackle the case of a family of pseudo-differential operators in dimension one, whose form appears for instance in \cite{BHR22} (where the tunneling effect is determined by subprincipal terms). However, the goal of the present article is not directly to recover the results of \cite{BHR22}. Its main objective is to explore optimal tunneling in a pseudodifferential context and pave the way for the analysis of purely magnetic tunneling in two dimension, see \cite{FBGMR2025}.

 Namely, we consider the pseudo-differential operator $\mathscr{L}_h$ given by
\[\mathscr{L}_h=a(hD_x)+h b(x,hD_x) \, ,\]
where $a$ and $b$ are real valued and belong to the symbol class 
\[S(\R^2)=\{p=p(x,\xi)\in\mathscr{C}^\infty(\R^2) : \forall\alpha\in\N^2\,,\exists C_\alpha>0 : |\partial^\alpha p(x,\xi)|\leq C_\alpha\}\,,\]
and where $p^w$ denotes the semiclassical Weyl quantization defined by
\[p^w\psi(x)=\left(\mathrm{Op}^w_h p\right)\psi(x)=\frac{1}{2\pi h}\int_{\R^2}e^{i(x-y)\eta/h}p\left(\frac{x+y}{2},\eta\right)\psi(y)\mathrm{d}y\mathrm{d}\eta\,.\]
The operator $\mathscr{L}_h$ is selfadjoint and bounded, from $L^2(\R)$ into $L^2(\R)$, in virtue of the Calder\'on-Vaillancourt theorem.

\subsection{Framework, heuristics and main result}

In the whole article, one will work under the following assumption on the principal symbol $a$.
\begin{assumption}\label{hyp:a}
	The real valued function $a\in S(\R^2)$ depends on $\xi$ only and we simply write $a(x,\xi)=a(\xi)$. It has a unique global minimum at $\xi=0$, assumed to be $0$, which is non-degenerate and not attained at infinity. Moreover, the function $a$ has a holomorphic extension to the strip $\Sigma_r=\R+i(-r,r)$, for some $r>0$. This extension, that we will still denote by $a$, belongs\footnote{$S(\Sigma_r)=\{p=p(x,\xi)\in\mathscr{C}^\infty(\Sigma_r) : \forall\alpha\in\N^2\,,\exists C_\alpha>0 : |\partial^\alpha p(x,\xi)|\leq C_\alpha\}$} to $S(\Sigma_r)$.
\end{assumption}
Near $\xi=0$, the symbol $a$ shares  the same features as the symbol of the Laplacian $\xi^2$. One could also consider $a$ in a slightly more general class containing $\xi^2$. To avoid the corresponding technicalities and to keep the proof as transparent as possible, we choose to focus our attention on the bounded case. Let us now describe the type of analytic perturbation that we want to deal with.

\begin{assumption}\label{hyp:b}
The function $b$ belongs to $S(\R^2)$ and it can be extended to $\R\times \Sigma_r$ holomorphically in the sense that $b(x,\cdot)\in \mathscr{O}(\Sigma_r)$ for all $x\in\R$ and that $b\in S(\R\times \Sigma_r)$.
\end{assumption}

We recall that we want to discuss the effect of symmetries on the spectrum of $p^w$. That is why we make the following assumption.

\begin{assumption}\label{hyp:c}
	The functions $a$ and $b$ are even in the sense that they satisfy $p(-X)=p(X)$ for all $X=(x,\xi)\in\R^2$. Moreover, $x\mapsto b(x,0)$ is non-negative and attains its minima at exactly two points $x_\ell<0$ and $x_r=-x_\ell>0$, which are non-degenerate minima. We assume that $b(x_\ell,0)=0$ and we let
	\[b_\infty=\liminf_{|x|\to+\infty}b(x,0)>0\,.\]
	\end{assumption}
Note that $a(\xi)=\xi^2$ (which doesn't satisfy Assumption \ref{hyp:a} though) and $b(x,\xi)=V(x)$ satisfy Assumption \ref{hyp:c} as soon as $V$ is a non-degenerate symmetric double well, which is covered by \cite{HS84}. 

Let us now discuss the heuristics that will allow us to guess and formulate a tunneling estimate for $\mathscr{L}_h$. We notice that
\[\mathscr{L}_h=\mathrm{Op}_{\hbar}^w (a(\hbar \xi)+hb(x,\hbar\xi))\,,\quad\mbox{ with }\hbar=h^{\frac12}\,,\]
which follows from a dilation in the integral defining the Weyl quantization. Naively, we write the formal expansion
 \[a(\hbar \xi)+hb(x,\hbar\xi)=h\left(\frac{a''(0)}{2}\xi^2+b(x,0)\right)+\mathcal{O}_\xi(h^{\frac32})\,.\]
If we forget the a priori non uniform remainder, we are reduced to a Schrödinger operator with double well potential. In particular this provides (see Section \ref{sec:3}) a rough estimate
\begin{equation}
 \lambda_1(\mathscr{L}_h) = \lambda_2(\mathscr{L}_h) + o\bigl(h^{\tfrac{3}{2}}\bigr) =  c_0 h^{\tfrac{3}{2}} + o\bigl(h^{\tfrac{3}{2}}\bigr), ~~ c_0=\sqrt{\frac{a''(0)\partial^2_{x,x}b(x_\ell,0)}{4}} \, . 
 \end{equation}

 We denote 
\[\mathscr{M}_\hbar=-\hbar^2\frac{a''(0)}{2}\partial^2_x+b(x,0)\,,\]
and we can recall the classical tunneling estimate (see, for instance, \cite{Harrell, Hel88, Robert} and the pedagogical paper \cite[Theorem 1.2]{BHR17}). The spectral gap between the lowest two eigenvalues satisfies
\begin{equation}\label{eq.gapMh}
\lambda_2(\mathscr{M}_\hbar)-\lambda_1(\mathscr{M}_\hbar)=(1+o(1))\mathsf{A}\hbar^{\frac12}e^{-\frac{\mathsf{S}}{\hbar}}\,,
\end{equation}
where
\[\mathsf{A}=4\left( \frac{a''(0)}{2}\right)^{\frac14} \left(\frac{\kappa}{\pi}\right)^\frac12 \sqrt{V(0)} \exp\left(-\int_{x_\ell}^{0} \frac{\partial_s\sqrt{V(s)}-\kappa}{\sqrt{V(s)}}\mathrm{d}s \right)\,,\]
and
\begin{equation*}
	\mathsf{S}=\sqrt{\frac{2}{a''(0)}}\int_{x_\ell}^{x_r} \sqrt{b(s,0)}\mathrm{d}s\,,\quad V(s) = b(s,0) \,,\quad\kappa =(\sqrt{V})'(x_\ell)\,.
\end{equation*}
Moreover, for some $c>0$, $\lambda_3(\mathscr{M}_\hbar)-\lambda_2(\mathscr{M}_\hbar)\geq c \hbar$.

Surprisingly, even though $h\mathscr{M}_\hbar$ is a rough approximation of our operator $\mathscr{L}_h$, the estimate \eqref{eq.gapMh} provides us with the one term asymptotics of the spectral gap for $\mathscr{L}_h$. In fact, our analyticity assumptions will allow us to deal with the remainders and describe the eigenfunctions and their exponential decay (see Section \ref{sec.orga} below).

Here is the main theorem.

\begin{theorem}\label{thm.main}
Under  Assumptions \ref{hyp:a}, \ref{hyp:b} and \ref{hyp:c}, we have
\[\lambda_2(\mathscr{L}_h) - \lambda_1(\mathscr{L}_h)\underset{h\to 0}{\sim}h(\lambda_2(\mathscr{M}_\hbar)-\lambda_1(\mathscr{M}_\hbar))\,.\]
Moreover, for some $c>0$, $\lambda_3(\mathscr{L}_h)-\lambda_2(\mathscr{L}_h)\geq c h^{\frac32}$.
\end{theorem}

\subsection{Organization and strategy}\label{sec.orga} 
We follow the same guidelines as in the "differential" case (see, for instance, the Bourbaki exposé \cite{Robert} and \cite[Chapter 6]{DS99}). On the one hand, we will see that a corner stone of the analysis is the Kuranishi trick stating that under suitable holomorphic assumptions on the symbol, a pseudodifferential operator conjugated by an exponential weight remains a pseudodifferential operator. This allows us to extend the Agmon estimates, which are of \emph{local} nature, and yields optimal WKB approximations adapted to the pseudo-differential context. On the other hand, the stationary phase theorem reveals an effective Schrödinger operator that makes our heuristics rigorous.

In Section \ref{sec.2}, we explain why the bottom of the spectrum of $\mathscr{L}_h$ is discrete, see Lemma \ref{spdisc}. Then, we start discussing the spectral analysis of the one-well operator $\mathscr{L}_{h,\ell}$ obtained after sealing the right well, see \eqref{eq.Lhl}. The main result in Section \ref{sec.2} is Proposition \ref{prop.quasimodes}, where we describe WKB quasimodes for $\mathscr{L}_{h,\ell}$. The proof of Proposition \ref{prop.quasimodes} is given in Section \ref{sec.22} and it relies on a classical WKB construction for pseudo-differential operators, which is itself based on the stationary phase theorem (see Appendices \ref{sec.A} \& \ref{sec.B}).

In Section \ref{sec:3}, we provide the reader with the one-term asymptotics of the low-lying eigenvalues of $\mathscr{L}_{h,\ell}$, see Proposition \ref{prop.simplewell}. In particular, we show that these eigenvalues are simple and separated by gaps of order $h^\frac32$. The proof relies on Proposition \ref{prop.quasimodes} and on a microlocalization lemma, namely Lemma \ref{lem.locamodes}, which shows that the eigenfunctions of $\mathscr{L}_{h,\ell}$ are microlocalized near $(x,\xi)=(x_\ell,0)$. The localization in $x$ is more subtle than the localization in $\xi$ since it originates from the behavior of the subprincipal symbol (which requires the use of the Fefferman-Phong inequality to be analyzed). These microlocalization results can be adapted to the two-well situation and they allow to prove a first estimate of the tunneling phenomenon, see Proposition \ref{prop.doublewell}. The spectrum of the double well operator is described, modulo $\mathcal{O}(h^\infty)$, as the union of the spectra associated with the one-well operators. The eigenvalues are distributed by duets, each duet being separated from the others by gaps of order $h^\frac32$.

In Section \ref{section:exp}, we improve the localization result near $x_\ell$ by establishing optimal Agmon estimates, see Proposition \ref{prop.Agmonfunctional}. This proposition is an elliptic estimate for the conjugated operator 	$\mathscr{L}_{h,\ell}^\varphi=e^{\frac{\varphi}{\sqrt{h}}}\mathscr{L}_{h,\ell}e^{-\frac{\varphi}{\sqrt{h}}}$. In Section \ref{sec.conse}, we apply Proposition \ref{prop.Agmonfunctional} to get optimal exponential decay estimates of the one-well eigenfunctions (see Proposition \ref{corollary.agmon}) and, most importantly, to a get a very accurate approximation of the eigenfunctions by the WKB quasimodes, see Corollary \ref{cor.wkb}. This approximation looks quite similar to that in the Schrödinger case (see  \cite[Proposition 2.7]{BHR17}). However, we emphasize that we are here in a pseudo-differential context (which is even rather degenerate) and that such good approximations are rare.

Section \ref{sec:inter} is devoted to the proof of Theorem \ref{thm.main}. To do so, we follow the method originally described in the Helffer-Sjöstrand papers. We first establish Proposition \ref{prop.splitformula} by showing that the space spanned by the WKB quasimodes is a good approximation of the space spanned by the eigenfunctions associated with the first two eigenvalues. The key step is Lemma \ref{lem.fstarquasi}.

Then, we study the interaction term, see Proposition \ref{prop.wh}. Here, the analysis deviates from the usual strategy consisting in representing the interaction term by means of an integral running over an interface between the wells (see \cite[Section 4.1]{BHR17} in dimension one). This strategy cannot be used in our pseudo-differential context since it is based on an integration by parts. Instead, we directly replace the one-well groundstates appearing in \eqref{eq:inter} by their WKB expansions and we use the stationary phase theorem. We end up with an integral that can be computed explicitely thanks to the transport equations determining the amplitude of the WKB Ansätze\footnote{Let us mention that it is also possible to follow the approach developed in \cite[pp.69-70]{HI} and in \cite[Section 3.2.3]{KG1994} to bypass the absence of the Green-Riemann formula. However, at least in the one dimensional case, it is more direct to use the transport equation defining the WKB quasimodes.}.

\section{The one-well operator and WKB constructions}\label{sec.2}
This section is devoted to the analysis of the left "one-well" operator $\mathscr{L}_{h,\ell}$, obtained by sealing the well on the right. This operator is defined as follows. We consider
\begin{equation}\label{eq.Lhl}
	\mathscr{L}_{h,\ell} := \mathscr{L}_h + k_\ell(x)\,,
	\end{equation}
where $k_\ell$ is a non-negative smooth function with support in $D(x_r,\eta)$ and such that the function $x\mapsto b_\ell(x,0)=b(x,0)+k_\ell(x)=V(x)+k_\ell(x)$ has a unique global minimum at $x_\ell$.
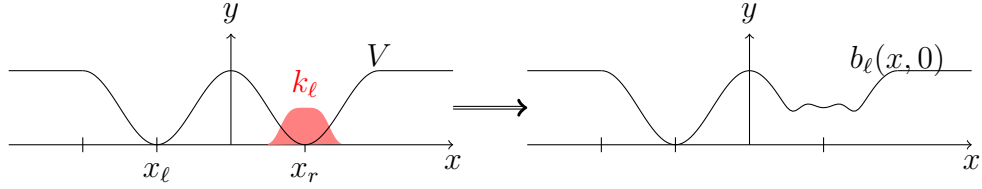
\begin{figure}[ht!]
	\begin{center}
		\begin{tikzpicture}[scale = 0.98]
			\draw[->] (-3,0) -- (3,0) node[below] {$x$};
			\draw[->] (0,0) -- (0,1.5) node[above] {$y$};
			\draw (-1,-0.1)--(-1,0.1);
			\draw (-2,-0.1)--(-2,0.1);
			\draw (1,-0.1)--(1,0.1);
			\draw (-1,-0.1) node[below] {$x_\ell$};
			\draw (1,-0.1) node[below] {$x_r$};
			\draw (2,1.5) node[below, color = black] {$V$};
			\draw (1,0.5) node[above, color=red] {$k_\ell$};
			\draw[domain=-3:-2,samples=100,color=black, thin] plot ({\x},{1});
			\draw[domain=2:3,samples=100,color=black, thin] plot ({\x},{1});
			\fill[color=red!50] (0.5,0) -- plot [domain=0.5:1.5, samples = 200] ({\x},{0.5*exp(-60*(\x-1)*(\x-1)*(\x-1)*(\x-1))})--cycle;
			\draw[double,->] (3,0.5) -- (4,0.5);
			\draw[domain=-2:2,samples=100,color=black, thin] plot ({\x+7},{0.5*(1-cos(180*(\x+1))))+0.5*exp(-60*(\x-1)*(\x-1)*(\x-1)*(\x-1))});
			\draw[domain=-3:-2,samples=100,color=black, thin] plot ({\x+7},{1});
			\draw[domain=2:3,samples=100,color=black, thin] plot ({\x+7},{1});
			\draw[domain=-2:2,samples=100,color=black, thin] plot ({\x},{0.5*(1-cos(180*(\x+1))))});
			\draw (9,1.5) node[below, color = black] {$b_\ell(x,0)$};
			\draw[->] (4,0) -- (10,0) node[below] {$x$};
			\draw[->] (7,0) -- (7,1.5) node[above] {$y$};
			\draw (6,-0.1)--(6,0.1);
			\draw (5,-0.1)--(5,0.1);
			\draw (8,-0.1)--(8,0.1);
		\end{tikzpicture}
	\end{center}
	\caption{Sealing a well}
	\label{fig:1}
\end{figure}

In a similar way, we define the operator on the right:
\[\mathscr{L}_{h,r} := \mathscr{L}_h + k_\ell(-x)\,.\]
Considering the symmetry operator 
  \begin{equation}\label{eq.U} \displaystyle U : \left \{ \begin{array}{ccc}
	\psi & \longmapsto & \psi(-\cdot) \\
	L^2(\R) & \longrightarrow & L^2(\R)
\end{array} \right.\,, \end{equation}
we observe that $\mathscr{L}_{h,\ell}=U^*\mathscr{L}_{h,r}U$. Thus, by unitary equivalence, all the results that we prove for $\mathscr{L}_{h,\ell}$ will hold for $\mathscr{L}_{h,r}$.

The following lemma, which is proved in Section \ref{sec.spdisc}, shows that the one-well operators have necessarily discrete spectrum below the threshold $b_\infty h$.
\begin{lemma}	\label{spdisc}
	Let $C\in\left(0,b_\infty\right)$. There exists $h_0>0$	such that, for all $h\in(0,h_0)$, the spectrum of $\mathscr{L}_{h,\ell}$ below $Ch$ is discrete.
\end{lemma}
The non-emptiness of the discrete spectrum follows from WKB contructions. We let
\begin{equation}
	\label{eq:phiell}
	\Phi_{\ell} : x \in \R \longmapsto \sqrt{\frac{2}{a''(0)}}  \left | \int_{x_\ell}^x \sqrt{b_{\ell}(s,0)}\mathrm{d}s \right | \, ,
\end{equation}
which is a $\mathscr{C}^{\infty}(\R)$ function since $s \longmapsto \mathrm{sgn}(s-x_\ell)\sqrt{b_\ell(s,0)}$ is smooth.

\begin{proposition}[WKB quasimodes]\label{prop.quasimodes}
	Let $n \in \N^*$. There exist
	\begin{enumerate}[\rm (i)]
	\item  a smooth function $u_{n}(\cdot,h)$ and a family $(u_{n,j})_{j \in \N}$ of elements of $S(\R)$ satisfying (in the sense of asymptotic expansions in $S(\R)$ as defined in Appendix \ref{apx.1})
	\[u_{n}(x,h)\underset{h\to 0}{\sim}\sum_{j\geq 0} u_{n,j}(x)h^{\frac{j}{2}}\,,\]
	where $u_{n,0}$ solves the transport equation
	\[ \left(i \Phi_\ell'(x)\partial_\xi b_\ell(x,0) + \frac{a''(0)}{2} \Phi_\ell''(x) +  a''(0)\Phi_\ell'(x) \partial_x-(2n-1)c_0\right)u_{n,0}=0\,,\]
	with $c_0=\sqrt{\frac{a''(0)\partial^2_{x,x}b(x_\ell,0)}{4}}>0$;
	\item a real number $\lambda^{\mathrm{WKB}}_n(h)$ and a family of real number $(\lambda_{n,j})_{j \in \N}$ satisfying
	\[\lambda^{\mathrm{WKB}}_{n}(h)\underset{h\to 0}{\sim}\sum_{j\geq 3} \lambda_{n,j}h^{\frac{j}{2}}\,,\quad \mbox{ with }\quad\lambda_{n,3}=(2n-1)c_0\,;\]
	\end{enumerate}
	such that the following holds. 
	
	Considering a smooth function $\chi$ with compact support\footnote{From now on, the space of such functions will be denoted $\mathscr{C}_0^{\infty}(\R)$} that  equals $1$ in a neighbourhood of $x_\ell$ and letting
	\begin{equation}\label{eq.wkbqm}
		\Psi^{\mathrm{WKB}}_{\ell,n}=h^{-\frac{1}{8}}\chi u_{n,h} e^{-\frac{\Phi_\ell}{\sqrt{h}}}\,,
	\end{equation}
	we have
	\[\left\|\left(\mathscr{L}_{h,\ell}-\lambda^{\mathrm{WKB}}_n(h)\right)\Psi^{\mathrm{WKB}}_{\ell,n}\right\|=\mathcal{O}\bigl(h^{\infty- \tfrac{n-1}{4}}\bigr)\lVert \Psi_{\ell,n}^{\mathrm{WKB}} \rVert \,.\]
	In particular, we have
	\[\lambda_n(\mathscr{L}_{h,\ell})\leq (2n-1)c_0h^{\frac32}+o(h^{\frac32})\,.\]
\end{proposition}

\subsection{Proof of Lemma \ref{spdisc}}\label{sec.spdisc}

	For all $z\leq C$, we have 
	\[\mathscr{L}_{h,\ell}-zh=\left[a(\xi)+h(b_\ell(x,\xi)-z)\right]^w\,.\]
For all $h, K >0$ satisfying $h K \leq 1$, since $a(\xi)^w$ is a positive operator, in the sense of quadratic forms,
\[ \mathscr{L}_{h,\ell} - zh \geq \left[h K a(\xi)+h(b_\ell(x,\xi)-z)\right]^w\, .\]	
Therefore for $K$ fixed large enough there exists $\psi(x,\xi) \in \mathscr{C}^{\infty}_0(\R^2)$ such that
\[\exists \epsilon >0, \quad K a(\xi) + b_\ell(x,\xi) - z + \psi(x,\xi) \geq \epsilon  \, .\]
This allows to apply the semiclassical Gärding inequality \eqref{eq:semgarding} to obtain
\[\mathscr{L}_{h,\ell} - zh + h\psi^w \geq \epsilon h/2. \]
Eventually the selfadjoint operator $\mathscr{L}_{h,\ell}-hz+h\psi^w$ is bijective. Since $\psi^w$ is compact, it follows that $\mathscr{L}_{h,\ell}-hz$ is Fredholm with index $0$. The discreteness of the spectrum below $Ch$ follows from the analytic Fredholm theory.

\begin{remark}
The proof of Lemma \ref{spdisc} can easily be adapted to the double well operator $\mathscr{L}_h$. Therefore, it indeed makes sense to study the spectral gap. 
\end{remark}
\subsection{Proof of Proposition \ref{prop.quasimodes}}\label{sec.22}

\subsubsection{WKB expansions}

Proposition \ref{prop.quasimodes} is essentially a consequence of the following, stronger one.

	\begin{proposition}[WKB constructions]\label{prop.WKB}
	 Let $n \in \N^*$. There exist a family of smooth functions $(u_{n,j})_{j\geq 0}$ and a family of real numbers $(\lambda_{n,j})_{j\geq 3}$ such that the following holds. Let $J\geq 1$ and $\chi \in \mathscr{C}_0^{\infty}(\R)$ equal to $1$ on a segment $I$ containing $x_\ell$. Letting
	\[ u^{[J]}_{n,h}=\sum_{j=0}^J h^{\frac{j}{2}}u_{n,j}\,,\qquad \lambda_n^{[J]}(h)=\sum_{j=3}^{J+3}\lambda_{n,j}h^{\frac{j}{2}}\,,\]
we have
	\begin{equation}\label{expan}
\left\| e^{\frac{\Phi_\ell}{\sqrt{h}}}\left(\mathscr{L}_{h,\ell}-\lambda_n^{[J]}(h)\right)\left(\chi u^{[J]}_{n,h} e^{-\frac{\Phi_\ell}{\sqrt{h}}}\right)\right\|_{ L^{\infty}(I)}=\mathcal{O}(h^{\frac{J+4}{2}})\,.
\end{equation}

Explicitely,
$\lambda_{n,3}=(2n-1)c_0$ and we can take
\begin{equation} 
\label{eq.princip}
\begin{aligned}
 u_{1,0}(x) & = \left( \frac{\Phi_\ell''(x_\ell)}{\pi} \right)^{1/4}\exp \left(-\int_{x_\ell}^x \left(\frac{\Phi_{\ell}''(s) - \Phi_{\ell}''(x_\ell)}{2\Phi_\ell'(s)} +\frac{i}{a''(0)}\partial_\xi b(s,0)\right)\mathrm{d}s \right). 
\end{aligned}
\end{equation}
\end{proposition}
\begin{proof}
In the following, we drop the reference to $n$.
Setting $\mathscr{L}^{\Phi_\ell}_{h,\ell} := e^{\frac{\Phi_{\ell}}{\sqrt{h}}} \mathscr{L}_{h,\ell} e^{-\frac{\Phi_{\ell}}{\sqrt{h}}}$, we want to find functions $u_j$ and numbers $\lambda_j$ such that
\begin{equation}
\label{WKB}
(\mathscr{L}^{\Phi_\ell}_{h,\ell} -\lambda^{[J]}(h))(\chi u^{[J]}_h) = \mathcal{O}_{L^\infty(I)}(h^{\frac{J+4}{2}})\,.
 \end{equation}
Thanks to Lemma \ref{lem:apxl}, the action of $\mathscr{L}_{h,\ell}^{\Phi_\ell}$ on $\mathscr{C}^\infty_0(\R)$ is that of a series of differential operators: there exists a family of differential operators $(P_{\gamma}(x,D_x))_{0 \leq \gamma \leq J+3}$ such that
\[\mathscr{L}_{h,\ell}^{\Phi_\ell} (\chi u_h^{[J]}) =  \sum_{0 \leq \gamma \leq J+3} h^{\gamma/2}P_{\gamma}(x,D_x)(\chi u_h^{[J]}) + \mathcal{O}_{L^\infty(\R)}(h^{\frac{J+4}{2}})\,,\]
where the first few operators are explicitely given by
\begin{equation}
\label{eq:opdiff}
\left \{ \begin{array}{lll}
P_0(x,D_x) = a(0) = 0, \\
P_1(x,D_x) = i \Phi_\ell'(x) a'(0) = 0, \\
P_2(x,D_x) = -\frac{a''(0)}{2} \Phi_\ell'(x)^2 + b_\ell(x,0) = 0, \\
 P_3(x,D_x)= i \Phi_\ell'(x)\partial_\xi b_\ell(x,0) + \frac{a''(0)}{2} \Phi_\ell''(x) +  a''(0)\Phi_\ell'(x) \partial_x\,. \end{array} \right.
\end{equation}
Therefore $(\mathscr{L}_{h,\ell}^{\Phi}-\lambda^{[J]}(h)) (\chi u_h^{[J]}) =  \mathcal{O}_{ L^\infty(I)}(h^{\frac{J+4}{2}})$ if and only if $\lambda_0 = \lambda_1 = \lambda_2 = 0$ and the $(\lambda_j,u_j)$ satisfy the transport equations
\begin{equation}
\label{WKBeq}
\left \{ \begin{array}{ll}
(P_3(x,D_x)-\lambda_3) u_0 = 0, \\
(P_3(x,D_x)-\lambda_3) u_j = b_j + \lambda_{j+3} u_0\,,\quad j \geq 1, \end{array} \right. 
\end{equation}
where $b_j := \displaystyle -P_{j+3}(x,D_x)u_0-\sum_{\gamma = 4}^{j+2} (P_{\gamma}(x,D_x)-\lambda_\gamma)u_{j-\gamma+3}$.

Let us start by solving the first transport equation, which can be written as
\begin{equation}
 \left(ir(\lambda_3,x) + \Phi_{\ell}''(x) \left( \frac{1}{2}-\frac{\lambda_3}{a''(0)\Phi_{\ell}''(x_\ell)} \right)+\Phi_{\ell}'(x) \partial_x\right)u_0 = 0\,,
\end{equation}
where 
\[ \displaystyle r(\lambda_3,x) = \frac{1}{a''(0)} \left( \Phi_\ell'(x)\partial_\xi b_\ell(x,0) -i \lambda_3 \frac{\Phi_{\ell}''(x)-\Phi_{\ell}''(x_\ell)}{\Phi_{\ell}''(x_\ell)} \right).\] It satisfies that $x \longmapsto \frac{r(\lambda_3,x)}{\Phi_{\ell}'(x)}$ is in $\mathscr{C}^{\infty}(\R)$. Since $\Phi_{\ell}'$ vanishes linearly at $x_\ell$, for every $\lambda_3 \in \R$, the differential equation $(P_3(x,D_x)-\lambda_3) u_0 = 0$ can be solved on $]x_{\ell},+\infty[$ and on $]-\infty,x_{\ell}[$. In both cases, the space of solutions is spanned by the function
\begin{equation}
\label{eq:vect}
 x \longmapsto \Phi_{\ell}'(x)^{\frac{\lambda_3}{a''(0)\Phi_{\ell}''(x_\ell)}-\frac 12} \exp\left(-i \int_{x_{\ell}}^x \frac{r(\lambda_3,s)}{\Phi_{\ell}'(s)}\mathrm{d}s\right) \,.
 \end{equation} 
 Therefore, there is a smooth solution on $\R$ if and only if there exists $n\in\mathbb{N}_{\geq 1}$ such that
\begin{equation}
\label{quantization}
 \frac{\lambda_3}{a''(0)\Phi_{\ell}''(x_\ell)}-\frac 12=n-1\,.
 \end{equation}
 From now on, we take $\lambda_3=\lambda_{n,3}=(2n-1)c_0$. With this choice, \eqref{eq:vect} becomes
\begin{equation}
\label{eq:monomial}
x \mapsto \Phi_{\ell}'(x)^{n-1} \exp\left(\frac{-i}{a''(0)}\int_{x_\ell}^x \partial_\xi b(s,0)\mathrm{d}s\right)\exp\left(-(2n-1) \int_{x_{\ell}}^x \frac{\Phi_\ell''-\Phi_\ell''(x_\ell)}{2\Phi_{\ell}'}\mathrm{d}s\right)\,. 
\end{equation}
The other transport equations in \eqref{WKBeq} become
\begin{equation} \label{transp}
\left(ir(\lambda_{n,3},x)-(n-1) \Phi_\ell''(x) + \Phi_{\ell}'(x) \partial_x \right) u_{n,j} = b_{n,j}+\lambda_{n,j+3} u_{n,0}\,.
\end{equation}
Letting $\widetilde{u}_{n,j} := u_{n,j} \exp \left(i \int_{x_{\ell}}^x \frac{r(\lambda_{n,3},s)}{\Phi_{\ell}'(s)}\mathrm{d}s \right)$ and $\widetilde{b}_{n,j} := b_{n,j} \exp \left(i \int_{x_{\ell}}^x \frac{r(\lambda_{n,3},s)}{\Phi_{\ell}'(s)}\mathrm{d}s \right)$, the previous equation \eqref{transp} becomes 
\begin{equation} \label{transp2}
\left(-(n-1) \Phi_\ell''(x) + \Phi_{\ell}'(x) \partial_x \right) \widetilde{u}_{n,j} = \widetilde{b}_{n,j}+\lambda_{n,j+3} \widetilde{u}_{n,0}\,.
\end{equation}

Here again, by the theory of ODEs, the solutions in $(-\infty,x_\ell)$ and $(x_\ell,+\infty)$ form one dimensional vector spaces. To prove that there exists global solutions it suffices to prove that there exists a $\mathscr{C}^{\infty}$ solution well defined near $x_\ell$.

The Taylor formula ensures the existence of $(\alpha_{n,j}^k)_{0\leq k\leq n-1}\in\R^n$ and of $r_{n,j} \in \mathscr{C}^{\infty}(\mathrm{Neigh}(x_\ell))$ such that, near $x_\ell$,
\[ \widetilde{b}_{n,j}(x) = \left(\alpha_{n,j}^0 + \alpha_{n,j}^1 (\Phi'_{\ell}(x))^2 + \cdots + \alpha_{n,j}^{n-1} (\Phi'_{\ell}(x))^{n-1} + r_{n,j}(x) (\Phi_{\ell}'(x))^n\right) \Phi_{\ell}''(x)\,. \]
Setting $\widetilde{v}_{n,j} = \sum_{k=0}^{n-2} \frac{\alpha_{n,j}^k}{k-(n-1)}(\Phi_\ell')^k$ and choosing $\lambda_{n,j+3} = -\alpha_{n,j}^{n-1}$, we get
\begin{equation}
\left(-(n-1) \Phi_\ell'' + \Phi_{\ell}' \frac{\mathrm{d}}{\mathrm{d}x} \right) (\widetilde{u}_{n,j}-\widetilde{v}_{n,j}) = (\Phi_{\ell}')^{n} \widetilde{r}_{n,j}\hbox{ with } \widetilde{r}_{n,j} \in \mathscr{C}^{\infty}(\mathrm{Neigh}(x_\ell)).
\end{equation}
 This last equation has smooth solutions of the form $(\Phi_{\ell}')^n \widetilde{s}_{n,j}$ in a neighbourhood of $x_\ell$. This proves that the equations (\ref{WKBeq}) have smooth solutions on $\R$ and concludes the proof.
\end{proof}

	\subsubsection{Proof of Proposition \ref{prop.quasimodes}}

	 For all $n \in \N_{\geq 1}$, the Borel lemma provides us with  
	\begin{enumerate}[---]
	\item a smooth function $u_{n}(\cdot,h)$ having the asymptotic expansion $\sum_{j\geq 0} u_{n,j}(x)h^{\frac{j}{2}}$, 
	\item a number $\lambda_n^{\mathrm{WKB}}(h)$  having the asymptotic expansion $\sum_{j\geq 3} \lambda_{n,j}h^{\frac{j}{2}}$,
	\end{enumerate}
	with the $u_{n,j}$ and $\lambda_{n,j}$ given by Proposition \ref{prop.WKB}. Let us recall that $\Psi_{\ell,n}^{\mathrm{WKB}}$ is defined in \eqref{eq.wkbqm}. We first deal with $n=1$,
		\begin{align*}
			\| \Psi^{\mathrm{WKB}}_{\ell,1} \|^2 = h^{-\frac14}\int_{\R} \chi(x)^2 |u_{1,h}(x)|^2 e^{-\frac{2\Phi_\ell(x)}{\sqrt{h}}}\mathrm{d}x\,.
		\end{align*}
We recall that, on the support of $\chi$, we have $u_{1,h}= u_{1,0} + \mathcal{O}(h^{\frac12})$ and we apply the Laplace method to get
		\begin{align*}
			\|\Psi_\ell^{\mathrm{WKB}}\|^2= \sqrt{\frac{\pi}{\Phi_\ell''(x_\ell)}} |u_{1,0}(x_\ell)|^2 + \mathcal{O}(h^\frac12) = 1 + \mathcal{O}(h^\frac12)\,,
		\end{align*}
		where we used \eqref{eq.princip} (especially the normalization constant). Thanks to Proposition \ref{prop.WKB}, we obtain
		 \[\left\|\left(\mathscr{L}_{h,\ell}-\lambda^{\mathrm{WKB}}_{1}(h))\right)\Psi^{\mathrm{WKB}}_{\ell}\right\|_{L^2(I)}=\mathcal{O}(h^{\infty})\,.\] 
		Then, we have
\[\begin{split}
		 	\left\|\left(\mathscr{L}_{h,\ell} \right. \right. -& \left. \left. \lambda^{\mathrm{WKB}}_{1}(h)\right)\Psi^{\mathrm{WKB}}_{\ell}\right\|_{L^{2}(\complement I)} =  \left\|e^{-\frac{\Phi_\ell}{\sqrt{h}}} (\mathscr{L}_{h,\ell}^{\Phi_\ell}-\lambda^{\mathrm{WKB}}_1(h)) (h^{-\frac18}\chi u_{1,h})\right\|_{L^{2}(\complement I)}\\
		 	&\leq\left\|e^{-\frac{\Phi_\ell}{\sqrt{h}}}\right\|_{L^\infty(\complement I)}\left\| (\mathscr{L}_{h,\ell}^{\Phi_\ell}-\lambda^{\mathrm{WKB}}_1(h)) (h^{-\frac18}\chi u_{1,h})\right\|_{L^{2}(\complement I)}=\mathcal{O}(h^\infty)\,,
		 	\end{split}\]
	where we used $\min_{\complement I} \Phi_\ell>0$ and Lemma \ref{lem:apxl} for the last estimate. By using the spectral theorem, this establishes Proposition \ref{prop.quasimodes} when $n=1$. 
	
	For $n\geq 2$, the estimate follows in the same way except for the estimate of the $L^2$-norm where we use Lemma \ref{lem:lap} (to deal with the fact that $u_{n,0}$ vanishes at $x_\ell$), which gives $ \| \Psi^{\mathrm{WKB}}_{\ell,n} \| \geq c h^{\tfrac{n-1}{4}}$ and proves
\[\left\|\left(\mathscr{L}_{h,\ell}-\lambda^{\mathrm{WKB}}_n(h)\right)\Psi^{\mathrm{WKB}}_{\ell,n}\right\|=\mathcal{O}\bigl(h^{\infty- \tfrac{n-1}{4}}\bigr)\lVert \Psi_{\ell,n}^{\mathrm{WKB}} \rVert \,.\]

\section{First tunneling estimate} \label{sec:3}
 This section is devoted to establishing the following two propositions.

\begin{proposition}[Microlocal harmonic approximation]\label{prop.simplewell}
	We have, for all $n\geq 1$,
	\[\lambda_n(\mathscr{L}_{h,\ell})=(2n-1)c_0h^{\frac32}+o(h^{\frac32})\,.\]	
\end{proposition}

\begin{proposition}[Rough tunneling estimate]\label{prop.doublewell}
We have
\[\lambda_2(\mathscr{L}_h)-\lambda_1(\mathscr{L}_h)=\mathcal{O}(h^\infty)\,, \quad \lambda_1(\mathscr{L}_h)=\lambda_1(\mathscr{L}_{h,\ell})+\mathcal{O}(h^\infty)\, ,\]
and for all $c\in(0,c_0)$, letting $h$ small enough, $\lambda_3(\mathscr{L}_{h})-\lambda_2(\mathscr{L}_h)\geq ch^{\frac32}$.
\[\,.\]
\end{proposition}

\subsection{Microlocal harmonic approximation}

\subsubsection{Microlocalization}
\begin{lemma}[Localization of eigenfunctions]\label{lem.locamodes}
	Let $n \in \N_{\geq 1}$ and consider a $n$-th normalized eigenfunction $\psi_{n,\ell}$ of $\mathscr{L}_{h,\ell}$. Let $\rho\in\mathscr{C}^\infty_0(\R)$ such that $\rho=1$ in a neighbourhood of $0$. We denote $\rho_\delta : \R \ni s \longmapsto \rho(h^{-\delta}s)$. Then, we have
\begin{align*}
& \forall \delta \in (0,2/3), ~ \psi_{n,\ell}  = \rho_\delta(\xi)^w \psi_{n,\ell} +\mathcal{O}(h^{\infty})\, ,\\
&\forall \delta \in (0,1/6), \quad	\psi_{n,\ell} = \rho_\delta(x-x_\ell)\psi_{n,\ell} +\mathcal{O}(h^{\infty})\,.
\end{align*}
\end{lemma}
\begin{proof}
In order to lighten the notations we will write, only in this proof, $\lambda_n(h) = \lambda_n(\mathscr{L}_{h,\ell})$ and recall that $\lambda_n(h) = \mathcal{O}(h^{\frac32})$. 
Let us begin with the following lower bound on $\mathscr{L}_{h,\ell}$.

\begin{lemma} \label{lem:inegop}
In terms of quadratic forms, for all $\nu >0$
\[ \mathscr{L}_{h,\ell}  \geq \bigl((1-h^\nu)a(\xi) + h b_\ell(x,0) - C h^{2-\nu} \bigr)^w.\]
\end{lemma}
\begin{proof}[Proof of the Lemma]
There exists $r \in S(\R^2)$ such that we have $b_\ell(x,\xi) = b_\ell(x,0) + a(\xi)^{1/2}r(x,\xi)$ with $a(\xi)^{1/2} = \mathrm{sgn}(\xi) \sqrt{a(\xi)} \in S(\R)$. Then for all $\nu >0$ , by the Cauchy-Schwarz inequality,
\begin{equation}
\begin{aligned}
\langle \mathscr{L}_{h,\ell} u , u \rangle & = \langle (a^w + hb_\ell(x,0) ) u , u \rangle + h \langle (a^{\frac12}r)^w u , u \rangle \\
& \geq \langle (a^w + hb_\ell(x,0) ) u , u \rangle + h \langle r^w u , (a^{\frac12})^w u \rangle + \mathcal{O}(h^2) \\
& \geq \langle (a^w + hb_\ell(x,0) ) u , u \rangle - h^{\nu} \| (a^{1/2})^w u \|^2 - \mathcal{O}(h^{2-\nu}) \| u \|^{2} \\& \geq \langle ((1-h^\nu)a^w + hb_\ell(x,0) ) u , u \rangle - \mathcal{O}(h^{2-\nu}) \| u \|^{2}. 
\end{aligned}
\end{equation}
\end{proof}

Let us first consider $\delta \in (0,\frac12)$ and prove the microlocalization in $\xi$. In order not to overload the notation, we write $\lambda(h) := \lambda_n(h)$ and $\psi := \psi_{n,\ell}$. Thanks to Lemma \ref{lem:inegop} and the estimate $\lambda_n(h) = \mathcal{O}(h^{\frac32})$ we infer
\begin{equation}
\label{eq:binfsymbol}
\exists c >0, ~ \forall u \in L^2(\R)\setminus\{ 0 \}, \quad \langle \bigl(\mathscr{L}_{h,\ell} -\lambda(h) + \rho_\delta^w \bigr) u , u \rangle \geq c h^{2\delta} \,  
. \end{equation}
We deduce that for $h$ small enough $\mathscr{L}_{h,\ell} -\lambda(h) + \rho_\delta^w$ is injective thus invertible (since it is Fredholm) and the Cauchy-Schwarz inequality yields for another $c>0$ \[ \| (\mathscr{L}_{h,\ell} -\lambda(h) + \rho_\delta^w)^{-1} \| \leq c h^{-2\delta}\, .\]
Let us follow the usual iterative procedure. We consider a sequence $(\rho_{n})_{n \in \N} \in \mathscr{C}_0^{\infty}(\R)^{\N}$ of nonnegative cut-offs decreasing on $\R_+$ and increasing on $\R_-$ such that $\rho_1 \succ \rho_2 \cdots \succ \rho_n$ in the sense 
\begin{equation}
\label{eq:defrhoj}
 \rho_j \succ \rho_{j+1} \Leftrightarrow
\mathrm{supp} ~ \rho_{j+1} \subset \{ \rho_j = 1 \} \hbox{ and } \rho_{j+1} = 1 \hbox{ near } 0, 
\end{equation}
in particular $(1-\rho_j)\rho_{j+1} = 0$. We denote $\rho_{j,\delta}(\xi) = \rho_j(h^{-\delta} \xi)$ and we infer by the pseudodifferential calculus (see, for instance, \cite[Theorems 4.17, 4.18 \& 4.24]{Zworski})
\begin{equation}
\begin{aligned}
0 & = (1-\rho_{0,\delta}^w)(\mathscr{L}_{h,\ell} - \lambda) \psi =  (\mathscr{L}_{h,\ell} - \lambda) (1-\rho_{0,\delta}^w) \psi + [\mathscr{L}_{h,\ell}, \rho_{0,\delta}^w]\psi \\ & =  (\mathscr{L}_{h,\ell} - \lambda+ \rho_{1,\delta}^w) (1-\rho_{0,\delta}^w) \psi + \mathcal{O}(h^{2-\delta})(1-\rho_{1,\delta}^w) \psi + \mathcal{O}(h^\infty).
\end{aligned}
\end{equation}
In the previous equation, the term $\mathcal{O}(h^{2-d})$ denotes an operator whose norm is bounded by $C h^{2-d}$ for some $C>0$. Since $\| (\mathscr{L}_{h,\ell} - \lambda + \rho_\delta^w)^{-1} \| \geq ch^{-2\delta}$ and $2-3\delta >0$, we get by induction, for all $j \in \N$,
\begin{equation}
\label{eq:locanew1}
\begin{split}
(1-\rho_{0,\delta}^w) \psi = \mathcal{O}(h^{2-3\delta}) (1-\rho_{1,\delta}^w)\psi + \mathcal{O}(h^\infty) = \mathcal{O}(h^{n(2-3\delta)}) (1-\rho_{1,\delta}^w)\psi + \mathcal{O}(h^\infty) \\ = \mathcal{O}(h^\infty).
\end{split}
\end{equation}
Note that it is also possible to extend the proof in order to obtain \eqref{eq:locanew1} for $\delta \in (0,\frac23)$. Letting $\delta \in (\frac12,\frac23)$, we write $\delta = \frac12 + \delta'$ and recall that a priori the pseudodifferential calculus works properly in $S_{\nu}(\R)$ for $0 \leq \nu < \frac12$.  However changing the semiclassical parameter by writing
$\rho_{\delta}(hD) = \mathrm{Op}_{h^\frac12}(\rho_{\delta'}(\xi))$
solves the issue. Since $\delta'< 1/4$ we still have \eqref{eq:binfsymbol} and 
by the pseudodifferential calculus (with semiclassical parameter $h^{1/2}$) we infer $\| [\mathscr{L}_{h,\ell}, \rho_{n,\delta}^w] \| = \mathcal{O}(h^{3/2-\delta'})$ and the end of the proof is similar.

Let us now consider the localization in $x$ and consider $\delta_x \in (0,1/6)$. We denote $\rho_{j,\delta_x}(x-x_\ell) = \rho_j(h^{-\delta_x}(x-x_\ell))$, $\rho_j$ being defined above \eqref{eq:defrhoj}. Thanks to Lemma \ref{lem:inegop} and the estimate $\lambda(h) = \mathcal{O}(h^{\frac32})$ we infer
\begin{equation}
\label{eq:resolvineg2}
(\mathscr{L}_{h,\ell} -\lambda+h \rho_{\delta_x}(x-x_\ell)) \geq c h^{1+2\delta},  
\end{equation}
and similarly as before the previous operator is invertible an satisfies the bound
$\| (\mathscr{L}_{h,\ell} -\lambda+h \rho_{\delta_x}(x-x_\ell))^{-1}\| \leq ch^{-(1+2\delta)}$ for some other $c>0$. We then have 
\begin{equation}
\label{eq:bootstrap2}
\begin{aligned}
0 & = (1-\rho_{0,\delta_x}(x-x_\ell))(\mathscr{L}_{h,\ell} - \lambda) \psi  \\ & =  (\mathscr{L}_{h,\ell} - \lambda- h\rho_{1,\delta_x}(x-x_\ell)) (1-\rho_{0,\delta_x}(x-x_\ell)) \psi + [\mathscr{L}_{h,\ell}, \rho_{0,\delta_x}(x-x_\ell)]\psi.
\end{aligned}
\end{equation}
The estimate of the bracket is a little less straighforward as we need to use the localization in $\xi$ to conclude. We consider $\delta_\xi \in (0,\tfrac{1}{2})$ then by the pseudodifferential calculus
\begin{equation}
\begin{aligned}
& [\mathscr{L}_{h,\ell}, \rho_{0,\delta}(x-x_\ell)] \psi = [a^w,\rho_{0,\delta}(x-x_\ell)]\psi + \mathcal{O}(h^{2-2\delta}) (1-\rho_{1,\delta_x}(x-x_\ell))\psi + \mathcal{O}(h^{\infty})\psi \\
& = [a^w,\rho_{0,\delta}(x-x_\ell)] \rho_{\delta_\xi}^w \psi + \mathcal{O}(h^{2-2\delta}) (1-\rho_{1,\delta_x}(x-x_\ell))\psi + \mathcal{O}(h^{\infty})\psi\\
& = -ih^{1-\delta_x} (a'(\xi)\rho_{0,\delta_x}'(x-x_\ell))^w \rho_{\delta_\xi}^w \psi + \mathcal{O}(h^{2-2\delta}) (1-\rho_{1,\delta_x}(x-x_\ell))\psi + \mathcal{O}(h^{\infty})\psi \\
& = -ih^{1-\delta_x} (a'(\xi)\rho_{0,\delta_x}'(x-x_\ell)\rho_{\delta_\xi}(\xi))^w \psi + \mathcal{O}(h^{2-2\delta}) (1-\rho_{1,\delta_x}(x-x_\ell))\psi + \mathcal{O}(h^{\infty})\psi\\
& = \mathcal{O}(h^{1+\delta_\xi - \delta_x})(1-\rho_{1,\delta_x}(x-x_\ell)) \psi + \mathcal{O}(h^{2-2\delta}) (1-\rho_{1,\delta_x}(x-x_\ell))\psi + \mathcal{O}(h^{\infty})\psi.
\end{aligned}
\end{equation}
Thanks to \eqref{eq:resolvineg2}, we infer from \eqref{eq:bootstrap2}
\[ (1-\rho_{0,\delta_x}(x-x_\ell)) \psi = \mathcal{O}(h^{\delta_\xi - 3 \delta_x}) (1-\rho_{1,\delta_x}(x-x_\ell)) \psi + \mathcal{O}(h^{\infty})\psi. \]
Taking $3\delta_x < \delta_\xi < \tfrac{1}{2}$ (which is possible since $\delta_x \in (0,\tfrac{1}{6})$) and proceeding similarly to \eqref{eq:locanew1} we conclude $(1-\rho_{0,\delta_x}(x-x_\ell)) \psi = \mathcal{O}(h^{\infty})$.
\end{proof}
\subsubsection{Proof of Proposition \ref{prop.simplewell}}
We only have to prove the lower bound. For that purpose, let us consider an orthonormal family of eigenfunctions $(\psi_j)_{1\leq j\leq n}$ associated with the eigenvalues 
$(\lambda_j(\mathscr{L}_{h,\ell}))_{1\leq j\leq n}$ and set
\[E=\underset{1\leq j\leq n}{\mathrm{span}}\,\psi_j\,.\]
We have, for all $\psi\in E$,
\[\langle\mathscr{L}_{h,\ell}\psi,\psi\rangle\leq \lambda_n(\mathscr{L}_{h,\ell})\|\psi\|^2\,.\]
Let us recall that by Lemma \ref{lem:inegop}, taking $\nu = \tfrac{1}{4}$,
\[ \mathscr{L}_{h,\ell} \geq \bigl((1-h^{1/4})a(\xi) + h b_\ell(x,0) - \mathcal{O}(h^{3/2+1/4}) \bigr)^w.\]
The localization Lemma \ref{lem.locamodes} yields that for $\delta = 1/4$, using $a(\xi) \rho_\delta(\xi) \geq \frac{a''(0)}{2}\xi^2(1-Ch^{\delta})\rho_\delta(\xi)$ for some $C>0$ (obtained by Taylor), letting $\psi \in E$,
\begin{equation}
\begin{aligned}
& \langle\mathscr{L}_{h,\ell}\psi,\psi\rangle  \geq \langle \Bigl(\bigl((1-h^{1/4})a(\xi)\rho_\delta(\xi) + h b_\ell(x,0)\bigr)^w  - \mathcal{O}(h^{3/2+1/4}) \Bigr) \psi, \psi \rangle \\ & \geq \langle \Bigl(\bigl((1-h^{1/4})\tfrac{a''(0)}{2}\xi^2(1-Ch^{\delta})\rho_\delta(\xi) + h b_\ell(x,0)  \bigr)^w - \mathcal{O}(h^{3/2+1/4}) \Bigr) \psi ,\psi \rangle \\ & \geq \langle \bigl((1-\widetilde{C} h^{1/4})\tfrac{a''(0)}{2}\xi^2 + h b_\ell(x,0) \bigr)^w \rho_\delta \psi , \rho_\delta \psi \rangle - \mathcal{O}(h^{3/2+1/4})\| \psi \|^2, ~~ \widetilde{C} >0.
\end{aligned}
\end{equation}
In other words, we have for some other $C>0$,
\[\forall \psi \in E, \quad	\langle\mathscr{L}_{h,\ell}\psi,\psi\rangle\geq Q^{\mathrm{elec}}_h(\rho_\delta^w\psi)-Ch^{3/2+1/4}\|\psi\|^2\,,\]
where $\delta = 1/4$ and
\[\forall \varphi \in H^2(\R), \quad Q^{\mathrm{elec}}_h(\varphi)=\frac{a''(0)}{2}(1-\widetilde{C}h^{\delta})\|hD_x\varphi\|^2+h\int_{\R}b_\ell(x,0)|\varphi|^2\mathrm{d}x\,.\]
Hence, for all $\psi\in E$,
\[Q^{\mathrm{elec}}_h(\rho_\delta^w\psi)\leq (\lambda_n(\mathscr{L}_{h,\ell})+Ch^{3/2+1/4})\|\rho_\delta^w\psi\|^2\,.\]
Due to Lemma \ref{lem.locamodes}, we have $\dim(\rho_\delta^w E)=n$ and the min-max theorem implies that
\[\lambda_n(Q_h^{\mathrm{elec}})\leq \lambda_n(\mathscr{L}_{h,\ell})+Ch^{3/2+1/4}\,.\]
Using the harmonic approximation of the eigenvalues of the electric Schrödinger operator \[ (2n-1)c_0 h^{3/2} + o(h^{3/2}) \leq \lambda_n(Q_h^{\mathrm{elec}}) \leq \lambda_n(\mathscr{L}_{h,\ell}) + o(h^{3/2})\,,\] 
the lower bound follows.

\subsection{Proof of Proposition \ref{prop.doublewell}}

\subsubsection{Localization}
Let us state a localization result for the eigenfunctions of the double well operator, which follows from the same arguments as in the proof of Lemma \ref{lem.locamodes}.
\begin{lemma}[Localization of $\mathscr{L}_h$]
\label{lem:locamodes2}
Let $M>0$ and $\rho \in \mathscr{C}_0^{\infty}(\R)$ such that $0 \leq \rho \leq 1$ and $\rho$ is equal to $1$ in a small neighbourhood of $0$. Letting $\delta \in (0,\frac16)$ we set, for all $x \in \R$, 
\[\underline{\rho_\delta}(x) = \rho(h^{-\delta}(x-x_\ell)) + \rho(h^{-\delta}(x-x_r))\,.\]  
For all $n\geq 1$, we have denoting $\psi_{h,n}$ a $n$-th normalized eigenfunction of $\mathscr{L}$,
\begin{equation}
\underline{\rho_\delta}\, \psi_{h,n} = \psi_{h,n} + \mathcal{O}(h^\infty)\,.
\end{equation}
\end{lemma}

Recalling the notation introduced at beginning of Section \ref{sec.2}, we consider the tensored operator
\[\mathscr{L}^{\mathrm{mod}}_h=\mathscr{L}_{h,\ell}\oplus\mathscr{L}_{h,r}\, \hbox{ acting on } L^2(\R) \oplus L^2(\R).\]	
Its low-lying spectrum is made of eigenvalues of multiplicity two:
\begin{equation}\label{eq.lambdamod}
\forall n \in \N_{\geq 1}\,,\quad\lambda_{2n-1}(\mathscr{L}^\mathrm{mod}_{h})=\lambda_{2n}(\mathscr{L}^\mathrm{mod}_{h}) = \lambda_{n}(\mathscr{L}_{h,\ell}).
\end{equation}

\subsubsection{End of the Proof of Proposition \ref{prop.doublewell}}
Let us consider the spaces
\[F_{h,n}=\mathrm{span} (\psi_{n,\ell},\psi_{n,r})\,,\]
which is of dimension two thanks to Lemma \ref{lem:locamodes2}.
Then, for all $\psi_h\in F_{h,n}$, we have
\[(\mathscr{L}_h-\lambda_{n}(\mathscr{L}_{h,\ell}))\psi_h=\mathcal{O}(h^\infty)\|\psi_h\|\,,\]
where we used the $\mathcal{O}(h^\infty)$-orthogonality of $(\psi_{n,\ell},\psi_{n,r})$.
From the spectral theorem, we deduce that there are at least two eigenvalues of $\mathscr{L}_h$ that are $\mathcal{O}(h^\infty)$-close to $\lambda_n(\mathscr{L}_{h,\ell})$. Thus,
\begin{equation}\label{eq.ublambdanLh}
	\forall n \in \N_{\geq 1}\,,\quad  \lambda_{n}(\mathscr{L}_h) \leq \lambda_n(\mathscr{L}^\mathrm{mod}_{h}) + \mathcal{O}(h^\infty)\,.
\end{equation}
For all $\psi \in L^2(\R)$, we let 
\[Q(\psi) := \langle \mathscr{L}_h \psi, \psi \rangle\,,\quad  Q_{\star}(\psi) := \langle \mathscr{L}_{h,\star} \psi, \psi \rangle \hbox{ for } \star = \ell,r\,.\]
For all $(\psi,\widetilde{\psi}) \in L^2(\R) \oplus L^2(\R)$, we let 
\[Q_{\oplus}(\psi,\widetilde{\psi}) := Q_{\ell}(\psi) + Q_r(\widetilde{\psi})\,.\]

Let us then consider an orthonormal family of eigenfunctions $(\psi_j)_{1\leq j\leq n}$ associated with the eigenvalues 
$(\lambda_j(\mathscr{L}_{h}))_{1\leq j\leq n}$ and set
	\[E=\underset{1\leq j\leq n}{\mathrm{span}}\,\psi_j\,.\]
We have, for all $\psi\in E$,
\begin{equation}\label{eq.ubQE}
Q(\psi) \leq \lambda_n(\mathscr{L}_h)\|\psi\|^2\,.
\end{equation}
We consider $\chi_\ell \in \mathscr{C}_0^{\infty}(\R)$ supported outside $(x_r - \eta,x_r+\eta)$ and satisfying $\chi_\ell = 1$ near $x_\ell$. This allows to define by symmetry $\chi_r = \chi_\ell(-\cdot)$ and $\chi = \chi_\ell + \chi_r$.
Thanks to Lemma \ref{lem:locamodes2}, we have, for all $\psi \in E$,
\[Q(\psi) \geq Q(\chi \psi) + \mathcal{O}(h^\infty)\|\psi\|^2\,,\]
and also $\chi_\ell\mathscr{L}_h\chi_r\psi=\mathcal{O}(h^\infty)\|\psi\|$. We infer that
\[Q(\psi) \geq Q(\chi \psi)+ \mathcal{O}(h^\infty)\|\psi\|^2\geq Q_{\oplus}(\chi_\ell \psi, \chi_r \psi)+ \mathcal{O}(h^\infty)\|\psi\|^2\,.\]
Combining this last estimate with \eqref{eq.ubQE}, we get
\[\lambda_n(\mathscr{L}_h) \geq \max_{\psi\in E\setminus \{0 \}}\frac{Q_{\oplus}(\chi_\ell\psi,\chi_r \psi)}{\| \psi \|^2} + \mathcal{O}(h^\infty)\,,\]
so that, again by Lemma \ref{lem:locamodes2}, we have
\[\lambda_n(\mathscr{L}_h) \geq  \max_{\psi\in E\setminus \{0 \}}\frac{Q_{\oplus}(\chi_\ell\psi,\chi_r \psi)}{(1+ \mathcal{O}(h^\infty)) (\| \chi_\ell \psi \|^2 + \| \chi_r \psi \|^2)} + \mathcal{O}(h^\infty)\,.\]
By noticing that $E\ni\psi\mapsto (\chi_\ell\psi,\chi_r\psi)$ is injective, we deduce that
\[\lambda_n(\mathscr{L}_h) \geq (1+\mathcal{O}(h^\infty))\min_{\substack{F \subset L^2(\R)\oplus L^2(\R) \\ \dim F=n}} \max_{(\psi,\widetilde{\psi}) \in F\setminus \{0\}}\frac{Q_{\oplus}(\psi,\widetilde{\psi})}{\| \psi \|^2 + \| \widetilde{\psi} \|^2 } + \mathcal{O}(h^\infty)\,.\]
By the min-max theorem, it follows that
\[\lambda_n(\mathscr{L}_h) \geq (1+ \mathcal{O}(h^\infty)) \lambda_{n}(\mathscr{L}_{h,\mathrm{mod}})+\mathcal{O}(h^\infty)\,.\]
Recalling \eqref{eq.ublambdanLh}, \eqref{eq.lambdamod}, and Proposition \ref{prop.simplewell}, this ends the proof.

\section{Exponential estimates of the eigenfunctions}
\label{section:exp}

Exponential decay estimates in the context of Schrödinger operators usually follow from the famous Agmon estimates (\cite{Agmon82, HS84}). These estimates can be interpreted as elliptic estimates for a conjugated operator (by a suitable exponential).

Let $\varphi \in \mathscr{C}^{\infty}(\R)$ such that $\varphi' = \varphi_1 + h^{\frac14} \varphi_2$ where $\varphi_1 \in S(\R)$, $\varphi_2 \in S_{\frac14}(\R)$.  
We will work  with the conjugated operator:
\begin{equation}\label{eq.Lphi}
	\mathscr{L}_{h,\ell}^\varphi=e^{\frac{\varphi}{\sqrt{h}}}\mathscr{L}_{h,\ell}e^{-\frac{\varphi}{\sqrt{h}}}\,.
\end{equation}

\subsection{A functional inequality}
\begin{proposition}\label{prop.Agmonfunctional}
	Let $C_0,R>0$ and $\varphi \in \mathscr{C}^{\infty}(\R)$ such that $\varphi' = \varphi_1 + h^{\frac14} \varphi_2$ where $\varphi_1 \in S(\R)$, $\varphi_2 \in S_{\frac14}(\R)$ satisfy $\varphi_1(x_\ell)=\varphi_2(x_\ell) = 0$ and
	\begin{equation}
		\label{eq:cond}
		\forall x \in \complement B(x_\ell,Rh^{1/4}), ~~ b_\ell(x,0)-\frac{a''(0)}{2}\varphi'(x)^2 \geq C_0 h^{\frac12}\,.
	\end{equation}
Then there exists $C > 0$ such that
	\begin{equation} 
		\label{eq:cond2}
		\forall x \in B(x_\ell,Rh^{\frac14})\,\quad |\varphi(x)| \leq Ch^{\frac12}\, , 
	\end{equation}
and for all $M < C_0$, there exist $c, h_0>0$ such that, for all $h\in(0,h_0)$, $\lambda< Mh^{\frac32}$ and $v\in L^2(\R)$,
	\begin{equation}
		h^{\frac32}\|v\|^2_{L^2(\complement B(x_\ell,Rh^{\frac14}))}\leq c\|(\mathscr{L}^{\varphi}_{h,\ell}-\lambda)v\|\|v\|+ch^{\frac32}\| v\|^2_{L^2(B(x_\ell,Rh^{\frac14}))}\,.
	\end{equation}
\end{proposition}

\subsubsection{An elliptic estimate}
The proof of Proposition \ref{prop.Agmonfunctional} is mainly a consequence of the following coercivity estimate, which is based on quadratic form manipulations like in the Schrödinger case.

\begin{lemma}\label{lem.coercivityelec}
There exist $h_0,\epsilon,C>0$ such that, for all $h\in(0,h_0)$, all $v\in L^2(\R)$ and all $\lambda\in \R$, we have 
\begin{equation}
\label{eq:coercivityelec}
	\Re\langle (\mathscr{L}^{\varphi}_{h,\ell}-\lambda) v, v\rangle\geq \langle (h b_\ell(x,0)-h\frac{a''(0)}{2}\varphi'(x)^2-\lambda) v, v\rangle + Ch^{\frac32+\epsilon} \lVert v \rVert_{L^2(\R)}^2\,.
\end{equation}
	
\end{lemma}
\begin{proof}
Thanks to Lemma \ref{lem:kur1}, we have
	\begin{equation}
	\label{eq:dvtphi}
	\mathscr{L}^{\varphi}_{h,\ell}=p_\varphi^w+\mathcal{O}(h^{3/2+ 1/4})\,,\quad p_\varphi(x,\xi)=a(\xi+ih^{\frac12}\varphi')+hb_\ell(x,\xi+ih^{\frac12}\varphi')\,.
\end{equation}
Since $\varphi'$ is bounded and by the Taylor formula, we find $r_h$ belonging to $S_{1/4}(\R^2)$ such that
\begin{equation}
\mathrm{Re}(p_\varphi)=a(\xi)+hb_\ell(x,\xi)-h\varphi'(x)^2 \frac{a''(\xi)}{2}+h^{2}r_h(x,\xi)\,.
\end{equation}
At this point, it is possible to conclude by Taylor expanding the symbol and using the Fefferman-Phong inequality (Theorem \ref{th:fph}). However, the following analysis proves that the use of Fefferman-Phong inequality can be replaced here by that of the Cauchy-Schwarz inequality.

We replace $b_\ell(x,\xi)$ with $b_\ell(x,\xi)-h\varphi'(x)^2 \frac{a''(\xi)}{2}$ in the proof of Lemma \ref{lem:inegop}. In particular 
\[ b_\ell(x,\xi)-h\varphi'(x)^2 \frac{a''(\xi)}{2} = b_\ell(x,0) - \varphi'(x)^2 \frac{a''(0)}{2}+ a(\xi)^{1/2} r(x,\xi), ~ r(x,\xi) \in S_{1/4,0}(\R), \]
where $S_{1/4,0}(\R)$ is defined in \eqref{eq.Sd1d2}. We infer by Cauchy-Schwarz that for $\nu \in (0,\tfrac{1}{4})$, 
\begin{equation}
\begin{aligned}
h\langle (a^{1/2}(\xi) r(x,\xi))^w u, u \rangle & = h\langle  r^w u, (a^{1/2})^w u \rangle - \mathcal{O}(h^{2-1/4})\|u\|^2 \\ & \geq - h^{\nu} \langle a^w u,u \rangle - h^{2-\nu} \| r^w u \|^2 - \mathcal{O}(h^{2-1/4})\|u\|^2 \\
& \geq - h^{\nu} \langle a^w u,u \rangle- \mathcal{O}(h^{2-1/4})\|u\|^2.
\end{aligned}
\end{equation}
We deduce
\begin{equation}
\mathrm{Re}(p_\varphi)^w \geq \bigl((1-h^{\nu}) a(\xi)+ h(b_\ell(x,0) - \varphi'(x)^2 \tfrac{a''(0)}{2})\bigr)^w- \mathcal{O}(h^{2-1/4}).
\end{equation}
Since the operator $a^w$ is nonnegative \eqref{eq:coercivityelec} follows from the previous equation and \eqref{eq:dvtphi}.
	\end{proof}

\subsubsection{Proof of Proposition \ref{prop.Agmonfunctional}}
Applying Lemma \ref{lem.coercivityelec} and using the condition on $\lambda$, we get
\begin{align*}
	h \int_{\R}\left(b_\ell(x,0)-\varphi'(x)^2 \frac{a''(0)}{2}-(M+ch^\epsilon )h^\frac12\right)| v|^2\mathrm{d}x\leq \Re\langle (\mathscr{L}^{\varphi}_{h,\ell}-\lambda)v,v\rangle \, .
\end{align*}
Then using the $\varphi' = \varphi_1 + h^{\frac14} \varphi_2$ where $\varphi_1 \in S(\R)$, $\varphi_2 \in S_{\frac14}(\R)$ satisfy $\varphi_1(x_\ell)=\varphi_2(x_\ell) = 0$, we infer that $|\varphi'|^2 = \mathcal{O}(1) h^\frac12$ on $\{ |x-x_\ell|\leq Rh^{\frac14} \}$ and in particular we deduce \eqref{eq:cond2}. This gives
\begin{multline*}
	h \int_{|x-x_\ell|\geq Rh^{\frac14}}\left(b_\ell(x,0)-\varphi'(x)^2 \frac{a''(0)}{2}-(M+ch^\epsilon )h^\frac12\right)| v|^2\mathrm{d}x\\
	\leq \Re\langle (\mathscr{L}^{\varphi}_{h,\ell}-\lambda) v, v\rangle+ch^{\frac32}\| v\|^2_{L^2(B(x_\ell,Rh^{\frac14}))} \,.
\end{multline*}
Then, \eqref{eq:cond} provides us with
\begin{equation*}
	h^{\frac32}\int_{|x-x_\ell|\geq Rh^{\frac14}}(C_0-(M+ch^\epsilon )))| v|^2\mathrm{d}x
	\leq \Re\langle (\mathscr{L}^{\varphi}_{h,\ell}-\lambda) v,v\rangle  +ch^{\frac32}\| v\|^2_{L^2(B(x_\ell,Rh^{\frac14}))}\,.
\end{equation*}
It remains to use the Cauchy-Schwarz inequality and the conclusion follows.

\subsection{Consequences}\label{sec.conse}
Let us now analyze the exponential decay of the eigenfunctions of $\mathscr{L}_{h,\ell}$. We recall that $\Phi_\ell$ is given by \eqref{eq:phiell}
\begin{equation}
	\Phi_{\ell} : x \in \R \longmapsto \sqrt{\frac{2}{a''(0)}}  \left | \int_{x_\ell}^x \sqrt{b_{\ell}(s,0)}\mathrm{d}s \right |\, .
\end{equation}
The behavior at infinity of $\Phi_\ell$ will not be important in the analysis. That is why we consider a bounded version of it, denoted by $\widetilde{\Phi}_\ell$. We take $A>0$ such that, for all $x\notin [-A,A]$, $\Phi_\ell(x)>\Phi_\ell(x_r)$ and define the following.

\begin{lemma}\label{lem.tildePhiell}
There exists a function $\widetilde{\Phi}_\ell\in S(\R)$ such that:
\begin{enumerate}[---]
\item $\widetilde{\Phi}_\ell=\Phi_\ell$ on $[-A,A]$,
\item $\widetilde{\Phi}_\ell$ is constant on $\complement[-2A,2A]$,
\item  $\widetilde{\Phi}_\ell(\pm2A)<\Phi_\ell(\pm2A)$,
\item for all $x\in\R$, $ (x-x_\ell)\widetilde\Phi_\ell'(x) \geq 0$,
\item $|\widetilde{\Phi}_\ell' | \leq |\Phi_\ell'|$ thus $\widetilde{\Phi}_\ell \leq \Phi_\ell$.
\end{enumerate}
\end{lemma}
\begin{proof}
We consider $\chi_0 \in \mathscr{C}_0^{\infty}(\R)$ such that $\chi_0 = 1$ on $[-A,A]$, $\chi_0 = 0$ on $\R \setminus [-2A,2A]$ and $0 \leq \chi_0 \leq 1$, and define
\[\forall x \in \R, \quad \widetilde{\Phi}_\ell(x) = \int_{x_\ell}^x \chi_0(s) \Phi_\ell'(s)\mathrm{d}s. \]
It is then straightforward that $\widetilde{\Phi}_\ell$ satisfies the desired requirements.
\end{proof}
Due to Proposition \ref{prop.simplewell}, there exist $h_0,M>0$ such that, for all $h\in(0,h_0)$,
\begin{equation}
\label{m}
\lambda_1(\mathscr{L}_{h,\ell}) \leq Mh^{\frac32}\,.
\end{equation}
In the following, we let $\mu(h)=\lambda_1(\mathscr{L}_{h,\ell})\,.$

\begin{corollary}[Agmon estimates]\label{corollary.agmon}
Let $\epsilon\in(0,1)$ and $\psi_{h}$ a normalized eigenfunction of $\mathscr{L}_{h,\ell}$ associated with the eigenvalue $\mu(h)$. There exist $C,h_0>0$ such that, for all $h\in(0,h_0)$,
\[\|e^{(1-\epsilon)	\widetilde\Phi_\ell/\sqrt{h}}\psi_h\|\leq C\|\psi_h\|\,.\]
\end{corollary}
\begin{proof}
We let $\varphi=\sqrt{1-\epsilon}\widetilde\Phi_\ell$. It satisfies the assumptions of Proposition \ref{prop.Agmonfunctional} with $C_0>M$ and $\varphi_2=0$. Indeed, we have
\[b_\ell(x,0)-\frac{a''(0)}{2}\varphi'(x)^2\geq b_\ell(x,0)-\frac{a''(0)}{2}(\Phi'_\ell)^2=\epsilon b_\ell(x,0)\,,\]
and there remains to use the quadratic behavior of $b_\ell(\cdot,0)$ near $x_\ell$ and to choose $R$ large enough to get \eqref{eq:cond}.

Then, we consider the eigenvalue equation
\[\left(\mathscr{L}_{h,\ell}-\mu(h)\right)\psi_h=0\,,\]
which is equivalent to
\[\left(\mathscr{L}^\varphi_{h,\ell}-\mu(h)\right)v=0\,,\quad v=e^{\frac{\varphi}{\sqrt{h}}}\psi_h\,,\]
where we recall the notation \eqref{eq.Lphi}. Applying Proposition \ref{prop.Agmonfunctional} to $v=e^{\frac{\varphi}{\sqrt{h}}}\psi_h$, we find that
\[\|v\|\leq C\|v\|_{L^2(B(x_\ell,Rh^{\frac14}))}\,.\]
Thus, by using that $\varphi/\sqrt{h}$ is bounded on $B(x_\ell,Rh^{\frac14})$, we get some constant $C>0$ such that 
	\[\|e^{\varphi/\sqrt{h}}\psi_h\|\leq C\|\psi_h\|\,.\]
	\end{proof}

Let us now turn to the WKB approximation of $\psi_h$. More precisely, we want to have an exponentially sharp approximation of $\psi_h$ on the interval $ K = [-A, x_r - \eta]$, for $\eta>0$ small enough. This will follow from Proposition \ref{prop.Agmonfunctional}. We consider $\underline{\chi}_\ell \in \mathscr{C}_0^{\infty}(\R)$ such that $\underline{\chi}_\ell \geq 0$, $\underline{\chi}_\ell = 1$ on $[ -2A,2A]$ and $\underline{\chi}_\ell = 0$ on $(-\infty,-3A) \cup (3A,+\infty)$.

We recall that our WKB quasimode is $\Psi_\ell^{\mathrm{WKB}}:= h^{-\frac18}\underline{\chi}_\ell e^{-\frac{\Phi_\ell}{\sqrt{h}}}u_{1,h}$, see Proposition \ref{prop.quasimodes}.

\begin{corollary}[WKB approximation]\label{cor.wkb}
Let $\Pi_{h,\ell}$ the projection on the groundstate of $\mathscr{L}_{h,\ell}$. Then, we have
\[\left \|e^{\frac{\Phi_\ell}{\sqrt{h}}}(\Psi^{\mathrm{WKB}}_{\ell}-\Pi_{h,\ell}\Psi_\ell^{\mathrm{WKB}}) \right \|_{L^2(K)} = \mathcal{O}(h^\infty)\,.\]
\end{corollary}
\begin{proof}
The idea to use Proposition \ref{prop.Agmonfunctional} with a weight $\varphi$ that is a refined version of $\sqrt{1-\epsilon}\widetilde{\Phi}_\ell$ close to $x_\ell$. For that purpose, for $R>0$, $N \in \N$, we let
\[ \varphi_h(x) := \widetilde{\Phi}_\ell(x)  - N h^{\frac12}\int_{x_\ell}^x \rho_{R,h}(s) \frac{\widetilde \Phi_\ell'(s)}{\widetilde \Phi_\ell(s)}\mathrm{d}s \,,\]
where the function $\rho_{R,h}$ is given by
\[\rho_{R,h}(s) := \rho\left(\frac{s-x_\ell}{Rh^{\frac14}}\right)\,,\] 
with $\rho \in \mathscr{C}_0^{\infty}(\R)$ satisfying $\rho \equiv 1$ on $\complement B(0,1)$, $\rho \geq 0$ and $\mathrm{supp }\,\rho \in \complement B(0,\frac12)$. We notice that
\[\varphi_h' = \widetilde\Phi_\ell' -h^\frac14 N\rho_{R,h} \frac{h^\frac14 \widetilde \Phi_\ell'}{\widetilde \Phi_\ell}\,,\] 
where $\widetilde{\Phi}_\ell' \in S(\R)$ and $\displaystyle N\rho_{R,h} \frac{h^\frac14 \widetilde{\Phi}_\ell'}{\widetilde{\Phi}_\ell} \in S_{\frac14}(\R)$.

Let us establish \eqref{eq:cond} with $C_0>M$ ($M$ being given in $\eqref{m}$).  Away from a small neighbourhood of $x_\ell$ denoted by $\mathrm{Neigh}(x_\ell)$, (independent of $h$), we have $\min_{\complement \mathrm{Neigh}(x_\ell)} \widetilde{\Phi}_{\ell} >0$ and $\max_{\complement \mathrm{Neigh}(x_\ell)} \widetilde{\Phi}_{\ell} <+\infty$. This implies, for $h$ small enough such that $h^\frac12 \frac{N}{\min_{\complement \mathrm{Neigh}(x_\ell)}\widetilde{\Phi}_{\ell}} < 1$, for all $x \in \complement \mathrm{Neigh}(x_\ell)$,
\begin{multline}
(\varphi_h'(x))^2 \leq \widetilde{\Phi}_\ell'^2(x)\left(1- h^\frac12 \frac{N}{\max_{\complement \mathrm{Neigh}(x_\ell)} \widetilde{\Phi}_{\ell}}\right)^2 \\ = \frac{2b_\ell(x,0)}{a''(0)}\left(1- 2h^\frac12 \frac{N}{\max_{\complement \mathrm{Neigh}(x_\ell)}\widetilde{\Phi}_\ell} + \mathcal{O}(h)\right)\,.
\end{multline}
Thus \eqref{eq:cond} is satisfied outside $\mathrm{Neigh}(x_\ell)$. Then, we choose $R = R(N)$ so that $\Phi_\ell(x_\ell \pm \frac12 Rh^{1/4})\geq Nh^{\frac12}$. Thus $\frac{Nh^\frac12}{\Phi_\ell} \leq 1$ on $\{ \rho_{R,h} = 1 \}$. Therefore, for all $x \in \complement B(x_\ell,Rh^{1/4}) \cap \mathrm{Neigh}(x_\ell)$,
\begin{align*}
b_\ell(x,0)- \frac{a''(0)}{2}\varphi_h'(x)^2 & \geq \frac{a''(0)}{2} \Phi_\ell'^2 \left(1-(1-\frac{Nh^\frac12}{\Phi_\ell})^2\right) \\ 
& \geq Nh^\frac12 \frac{a''(0)}{2} c \left( 2 - \frac{Nh^\frac12}{\Phi_\ell} \right) \\
& \geq Nh^\frac12 \frac{a''(0)}{2} c \,,
\end{align*}
where we used that $\frac{(\Phi'_\ell)^2}{\Phi_\ell}$ is continuous near $x_\ell$ and bounded from below by a constant $c >0$. Taking $N$ large enough the condition \eqref{eq:cond} is satisfied with $C_0 > M$. \\

Let us consider $v=e^{\frac{\varphi_h}{\sqrt{h}}}(\Psi^{\mathrm{WKB}}_{\ell}-\Pi_h\Psi_\ell^{\mathrm{WKB}})$, which satisfies 
\[ (\mathscr{L}_{h,\ell}^{\varphi_h}-\mu(h))v = e^{(\varphi_h - \Phi_\ell)/\sqrt{h}}(\mathscr{L}_{h,\ell}^{\Phi_\ell}-\mu(h))h^{-\frac{1}{8}} \underline{\chi}_\ell u_{1,h}\,.\]
Recalling that $\underline{\chi}_\ell = 1$ on $[-2A,2A]$, by Proposition \ref{prop.WKB}, the Borel lemma and using $\varphi_h \leq \Phi_\ell$ on $[-2A,2A]$, we get
\begin{equation}
\lVert (\mathscr{L}_{h,\ell}^{\varphi_h}-\mu(h))v \rVert_{L^2(-2A,2A)} = \mathcal{O}(h^{\infty})\,.
\end{equation}
Moreover, by using Lemma \ref{lem.tildePhiell}, especially because \[ \forall s \in \complement (-2A,2A), \quad \widetilde{\Phi}_\ell(s)-\Phi_\ell(s) < \max_{\pm} \left(\widetilde{\Phi}_\ell(\pm2A)-\Phi_\ell(\pm2A)\right)<0\, ,\]
we obtain
\begin{equation*}
\begin{split}
	\lVert (\mathscr{L}_{h,\ell}^{\varphi_h}-\mu(h)))v & \rVert_{L^2(\complement(-2A,2A))} = \lVert e^{\frac{\varphi_h-\Phi_\ell}{\sqrt{h}}} (\mathscr{L}_{h,\ell}^{\Phi_\ell} - \mu(h)) (h^{-\frac18}\underline{\chi}_\ell u_{1,h}) \rVert_{L^2(\complement(-2A,2A))} \\
& \leq  |e^{\frac{\widetilde{\Phi}_\ell-\Phi_\ell}{\sqrt{h}}}|_{L^{\infty}(\complement(-2A,2A))} \| \mathscr{L}_{h,\ell}^{\Phi_\ell} -\mu(h)\|_{\mathcal{L}(L^2(\R))} \|h^{-\frac18}\underline{\chi}_\ell u_{1,h} \|  \\
&= \mathcal{O}(h^{\infty})\,.
\end{split}
\end{equation*}
This gives
\[\lVert (\mathscr{L}_{h,\ell}^{\varphi_h}-\lambda_1(\mathscr{L}_{h,\ell})) v\rVert = \mathcal{O}(h^\infty)\,.\]
We apply Proposition \ref{prop.Agmonfunctional} to get
\begin{equation}\label{eq.boundv}
\|  v \|  \leq \widetilde{C} \| v \|_{L^2(B(x_\ell,Rh^{1/4}))}+\mathcal{O}(h^\infty)\,.
\end{equation}
Moreover, by the spectral theorem, then Propositions \ref{prop.simplewell} and \ref{prop.quasimodes}, we get
\begin{align*}
\lVert \Psi_{\ell}^{\mathrm{WKB}} - \Pi_h \Psi_{\ell}^{\mathrm{WKB}} \|  & \leq (\lambda_{2}(\mathscr{L}_{h,\ell}) -\lambda_{1}(\mathscr{L}_{h,\ell}))^{-1} \| (\mathscr{L}_{h,\ell} - \lambda_1(\mathscr{L}_{h,\ell})) \Psi_\ell^{\mathrm{WKB}} \|  \\ 
& = (\lambda_{2}(\mathscr{L}_{h,\ell}) -\lambda_{1}(\mathscr{L}_{h,\ell}))^{-1} \mathcal{O}(h^\infty)  = \mathcal{O}(h^\infty) \,.
\end{align*} 

By using \eqref{eq.boundv} and \eqref{eq:cond2}, we deduce that
\begin{align*}
\lVert e^{\frac{\varphi_h}{\sqrt{h}}} (\Psi_\ell^{\mathrm{WKB}} - \Pi_{h,\ell} \Psi_\ell^{\mathrm{WKB}}) \rVert & \leq \widetilde{C} \|e^{\frac{\varphi_h}{\sqrt{h}}}( \Psi_\ell^{\mathrm{WKB}} - \Pi_{h,\ell} \Psi_\ell^{\mathrm{WKB}})\|_{L^2(B(x_\ell,Rh^{1/4}))} +\mathcal{O}(h^\infty)\\
& \leq \widetilde{C} e^{C} \|\Psi_\ell^{\mathrm{WKB}} - \Pi_{h,\ell} \Psi_\ell^{\mathrm{WKB}}\|_{L^2(B(x_\ell,Rh^{1/4}))} +\mathcal{O}(h^\infty)\\ & =  \mathcal{O}(h^\infty). 
\end{align*}
Then, we find that
\begin{align*}
\mathcal{O}(h^{\infty}) & = \| e^{\frac{\varphi_h}{\sqrt{h}}} (\Psi_\ell^{\mathrm{WKB}} - \Pi_{h,\ell} \Psi_\ell^{\mathrm{WKB}}) \| \geq \| e^{\frac{\varphi_h}{\sqrt{h}}} (\Psi_\ell^{\mathrm{WKB}} - \Pi_{h,\ell} \Psi_\ell^{\mathrm{WKB}}) \|_{L^2(K)} \\ & \geq \inf_{K} e^{\frac{\varphi_h - \Phi_\ell}{\sqrt{h}}}  \lVert e^{\frac{\Phi_\ell}{\sqrt{h}}} (\Psi_\ell^{\mathrm{WKB}} - \Pi_{h,\ell} \Psi_\ell^{\mathrm{WKB}}) \rVert_{L^2(K)} \,.
\end{align*} 
There remains to notice that
\[e^{\frac{\varphi_h - \Phi_\ell}{\sqrt{h}}} \geq \left|\frac{|\Phi_\ell|_{L^{\infty}(K)}}{\Phi_\ell(x_\ell \pm Rh^{1/4}/2)} \right|^{-N} \underset{h\to 0}{\sim} \left|\frac{8|\Phi_\ell|_{L^{\infty}(K)}}{\Phi_\ell''(x_\ell)R^2}\right|^{-N} h^{\frac{N}{2}},\]  
which is absorbed by the $\mathcal{O}(h^\infty)$ in the left-hand-side. This concludes the proof.
\end{proof}

\section{The interaction term} \label{sec:inter}
In this section, we prove Theorem \ref{thm.main}. We recall that $U$ is given in \eqref{eq.U}. We consider $\chi_\ell \in \mathscr{C}_0^{\infty}(\R)$ such that $\chi_\ell \geq 0$, $\chi_\ell = 1$ on $[-A,x_r - 2\eta]$ and $\chi_\ell = 0$ on $(-\infty,-2A) \cup (x_\ell-\eta,+\infty)$. We let $\chi_r = U \chi_\ell$, $\psi_{h,r}=U\psi_{h,\ell}$. The function $\psi_{h,r}$ is a groundstate of $\mathscr{L}_{h,r}$. We also set 
 \[f_{h,\ell} = \chi_\ell \psi_{h,\ell}\,,\qquad f_{h,r} =U f_{h,\ell}\,.\] 
 
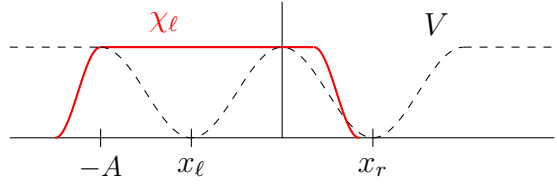
\begin{figure}[h]
	\begin{center}
		\begin{tikzpicture}[scale = 1.2]
			\draw (-3,0) -- (3,0);
			\draw (0,0) -- (0,1.5);
			\draw (-1,-0.1)--(-1,0.1);
			\draw (-2,-0.1)--(-2,0.1);
			\draw (1,-0.1)--(1,0.1);
			\draw (-1,-0.1) node[below] {$x_\ell$};
			\draw (1,-0.1) node[below] {$x_r$};
			\draw (-2,-0.1) node[below, scale = 0.9] {$-A$};
			\draw[color = red] (-1.3,1.5) node[below] {\bf $\chi_\ell $};
			\draw (1.7,1.5) node[below, color = black, thin] {$V $};
			\draw[domain=-2:0.34,samples=10,color=red, thick] plot ({\x},{1});
			\draw[domain=-2.5:-2,samples=100,color=red, thick] plot ({\x},{0.5*(sin(2*180*(\x-4.75))+1)});
			\draw[domain=0.35:0.85,samples=50,color=red, thick] plot ({\x},{0.5*(1+sin(2*180*(\x+0.9)))});
			\draw[domain=-2:2,samples=100,color=black, thin, dashed] plot ({\x},{0.5*(1-cos(180*(\x+1))))});
			\draw[domain=-3:-2,samples=100,color=black, thin, dashed] plot ({\x},{1});
			\draw[domain=2:3,samples=100,color=black, thin, dashed] plot ({\x},{1});
		\end{tikzpicture}
	\end{center}
	\caption{The function $\chi_\ell$}
	\label{fig:2}
\end{figure} 
 
Recalling Proposition \ref{prop.quasimodes} we also denote $u_{h,\ell} = u_{1,h}$ and $u_{h,r} = Uu_{h,\ell}$. Our WKB Ansätze are, with $\underline{\chi}_\ell$ defined in the previous section (and $\underline{\chi}_r = \underline{\chi}_\ell(-\cdot)$)
\[\Psi_\ell^{\mathrm{WKB}} =  h^{-\frac18}\underline{\chi}_\ell e^{-\frac{\Phi_\ell}{\sqrt{h}}} u_{h,\ell}\,,\qquad \Psi_r^{\mathrm{WKB}} =U\Psi_\ell^{\mathrm{WKB}}=   h^{-\frac18}\underline{\chi}_r e^{-\frac{\Phi_r}{\sqrt{h}}} u_{h,r}\,.\]
Since $\| \Psi_\ell^{\mathrm{WKB}} \| = 1 + \mathcal{O}(h^\frac12)$, we can write 
\begin{equation}
	\label{eq.ch}
	\Pi_{h,\ell} \Psi_\ell^{\mathrm{WKB}} = c(h) \psi_{h,\ell} \hbox{ with }|c(h)| = 1 + \mathcal{O}(h^\frac12)\, ,
\end{equation} 
and, by symmetry, $\Pi_{h,r} \Psi_r^{\mathrm{WKB}} = c(h) \psi_{h,r}$. Let us also denote by $\psi_{h,1},\psi_{h,2}$ an orthonormalized pair of eigenfunctions of $\mathscr{L}_h$ associated with its first two eigenvalues.  We also consider the orthogonal projection $\Pi_h : L^2(\R) \longrightarrow \mathrm{span}(\psi_{h,1},\psi_{h,2})$.

\subsection{The interaction formula}
The aim of this section is to establish the following proposition.
\begin{proposition}\label{prop.splitformula}
	We have
	\begin{equation}
	\label{eq:inter}
	\lambda_2(\mathscr{L}_h)-\lambda_1(\mathscr{L}_h)=2|w_h|+\widetilde{\mathcal{O}}(e^{-\frac{2S}{\sqrt{h}}})\,,
	\end{equation}
	with $\mu(h) = \lambda_1(\mathscr{L}_{h,\ell})$, $w_h=\langle(\mathscr{L}_h-\mu(h))f_{h,\ell},f_{h,r}\rangle$ and $S = \int_{x_\ell}^{x_r} \sqrt{b(s,0)}\,\mathrm{d}s$.	
\end{proposition}

We recall that
$r_h=\widetilde{\mathcal{O}}(e^{-\frac{S}{\sqrt{h}}})$ means that for every $\gamma >0$, we have $r_h=\mathcal{O}(e^{-\frac{S-\gamma}{\sqrt{h}}})$ up to adjusting $\eta$.

\subsubsection{The space $\mathrm{span}(f_{h,\ell},f_{h,r})$ is a good approximation of  $\mathrm{span}(\psi_{h,1},\psi_{h,2})$}
The groundstates of $\mathscr{L}_{h,\ell}$ and $\mathscr{L}_{h,r}$ are good quasimodes for the double well operator $\mathscr{L}_h$, with remainder of order  $\widetilde{\mathcal{O}}(e^{-\frac{S}{\sqrt{h}}})$.

\begin{lemma}\label{lem.psistarquasi}
For $\star = r,\ell$, we have
\begin{equation*}
(\mathscr{L}_h -\lambda_1(\mathscr{L}_{h,\star})) \psi_{h,\star} = \widetilde{\mathcal{O}}(e^{-\frac{S}{\sqrt{h}}})\,.
\end{equation*}
\end{lemma}
\begin{proof}
By symmetry, it suffices to prove for $\star = \ell$. We recall that 
 \[\mathscr{L}_h \psi_{h,\ell} = \lambda_1(\mathscr{L}_{h,\ell}) \psi_{h,\ell} - k_\ell \psi_{h,\ell}\,,\] 
 where $k_{\ell}$ has support in $D(x_r,\eta)$.

Thanks to Corollary \ref{corollary.agmon}, we have 
\[ \| k_\ell \psi_{h,\ell} \| \leq |k_{\ell}|_{L^{\infty}}e^{-\frac{(1-\epsilon) (S +\mathcal{O}(\eta))}{\sqrt{h}}} \left \| e^{(1-\epsilon)\frac{\widetilde{\Phi}_\ell}{\sqrt{h}}} \psi_{h,\ell} \right \| \leq C e^{-\frac{(1-\epsilon) (S + \mathcal{O}(\eta))}{\sqrt{h}}} \| \psi_{h,\ell} \|\,,\]
where we used that  $\| \Phi_\ell(x) - S \|_{L^{\infty}(x_r-\eta,x_r+\eta)} = \mathcal{O}(\eta)$. Up to choosing
$\epsilon,\eta >0$ small enough, we have the desired estimate.
\end{proof}
In fact, the same estimate holds if we insert the cutoff functions $\chi_\star$.

\begin{lemma}\label{lem.fstarquasi}
For $\star = r,\ell$, we have
\[(\mathscr{L}_h - \lambda_1(\mathscr{L}_{h,\ell}))f_{h,\star} = \widetilde{\mathcal{O}}(e^{-\frac{S}{\sqrt{h}}})\,,\quad (\mathrm{Id}-\Pi_h)f_{h,\star}=\widetilde{\mathcal{O}}(e^{-\frac{S}{\sqrt{h}}})\,.\]
Moreover, we have $\langle f_{h,\ell}, f_{r,h}\rangle=\widetilde{\mathcal{O}}(e^{-\frac{S}{\sqrt{h}}})$ and
$\lambda_j(\mathscr{L}_{h})=\mu(h)+\widetilde{\mathcal{O}}(e^{-\frac{S}{\sqrt{h}}})$
\end{lemma}
\begin{proof}
 Using again Corollary \ref{corollary.agmon},
\begin{equation*}
 \| (1-\chi_\ell)\psi_{h,\ell} \| \leq \| e^{-(1-\epsilon)\widetilde{\Phi}_\ell/\sqrt{h}} (1-\chi_\ell) e^{(1-\epsilon)\widetilde{\Phi}_\ell/\sqrt{h}}\psi_{h,\ell} \|
 = \widetilde{\mathcal{O}}(e^{-\frac{S}{\sqrt{h}}}). 
 \end{equation*}
Combining this with the boundedness of $\mathscr{L}_h$, we get
\[ (\mathscr{L}_h - \lambda_1(\mathscr{L}_{h,\ell}))(\psi_{h,\ell}-f_{h,\ell})= \widetilde{\mathcal{O}}(e^{-\frac{S}{\sqrt{h}}}). \]
Using Lemma \ref{lem.psistarquasi}, this gives $(\mathscr{L}_h - \lambda_1(\mathscr{L}_{h,\ell}))f_{h,\star} = \widetilde{\mathcal{O}}(e^{-\frac{S}{\sqrt{h}}})$ when $\star=\ell$. The case $\star=r$ follows by symmetry.

Corollary \ref{corollary.agmon} shows that $\langle f_{h,\ell}, f_{h,r}\rangle=\widetilde{\mathcal{O}}(e^{-\frac{S}{\sqrt{h}}})$ (mainly because $\widetilde{\Phi}_\ell + \widetilde{\Phi}_r = S$ on $[x_\ell,x_r]$). Therefore, the spectral theorem and Proposition \ref{prop.doublewell} show that, for $j=1,2$,
\[ \lambda_j(\mathscr{L}_{h})=\mu(h)+\widetilde{\mathcal{O}}(e^{-\frac{S}{\sqrt{h}}})\,,\]
and also
\[(\mathrm{Id}-\Pi_h)f_{h,\star}=\widetilde{\mathcal{O}}(e^{-\frac{S}{\sqrt{h}}})\,.\]
\end{proof}

\subsubsection{Proof of Proposition \ref{prop.splitformula}}

We let
\[g_{\star}=\Pi_h f_{h,\star}\,,\]
and
\[G=\begin{pmatrix}
	\langle g_\ell,g_\ell\rangle&\langle g_\ell,g_r\rangle\\
	\langle g_r,g_\ell\rangle&\langle g_r,g_r\rangle
\end{pmatrix}=\begin{pmatrix}
	g_\ell\\
	g_r
\end{pmatrix}\cdot\begin{pmatrix}
	g_\ell&g_r
\end{pmatrix}\geq 0\,.\]
We consider the matrix of the quadratic form associated with $\mathscr{L}_h$ in the basis $(g_\ell,g_r)$,
\[L=\begin{pmatrix}
	\langle \mathscr{L}_hg_\ell,g_\ell\rangle&\langle  \mathscr{L}_hg_\ell,g_r\rangle\\
	\langle  \mathscr{L}_hg_r,g_\ell\rangle&\langle  \mathscr{L}_hg_r,g_r\rangle
\end{pmatrix}.\]
The following proposition is a classical consequence of Lemma \ref{lem.fstarquasi}, cf. the concise presentation in \cite[Section 4.1]{FMR23}
\begin{proposition}
The family $(g_\ell,g_r)$ is asymptotically an orthonormal basis of $F$, in the sense that
		\begin{equation}\label{eq.G}
			G=\mathrm{Id}+\mathrm{T}+\widetilde{\mathcal{O}}(e^{-2S/\sqrt{h}})\,,\quad \mathrm{T}=\begin{pmatrix}
				0&	\langle f_{h,\ell},f_{h,r}\rangle\\
				\langle f_{h,r},f_{h,\ell}\rangle&0
			\end{pmatrix}=\widetilde{\mathcal{O}}(e^{-\frac{S}{\sqrt{h}}})\,.
		\end{equation}
		In addition, we have
		\begin{equation}\label{eq.L}
			L=\begin{pmatrix}
				\mu(h)& w_h\\
				\overline{w}_h&\mu(h)
			\end{pmatrix}+\mu(h)\mathrm{T}+\widetilde{\mathcal{O}}(e^{-2S/\sqrt{h}})\,,\quad w_h=\langle(\mathscr{L}_h-\mu(h))f_{h,\ell},f_{h,r}\rangle=\widetilde{\mathcal{O}}(e^{-\frac{S}{\sqrt{h}}})\,.
		\end{equation}

	\end{proposition}
	
\begin{proof}[Adaptation of the proof of \cite{FMR23} Section 4.1]
To be fully correct, in \cite[Section 4.1]{FMR23} the authors use the IMS formula that is not valid in the pseudodifferential context to derive the estimates of $\langle \mathscr{L}_h g_\star, g_\star \rangle$. However their argument can be naturally adapted to the present context in the spirit of the proof of Lemma \ref{lem.fstarquasi}. In order not to overload the presentation, let us only provide the reader with the adaptation of the computation of $\langle \mathscr{L}_h g_\ell, g_\ell \rangle.$ First note that $\mathscr{L}_h$ decomposes as the orthogonal sum $\mathscr{L}_h \Pi + \mathscr{L}_h(\mathrm{Id} - \Pi)$ thus $\mathscr{L}_h \Pi = \mathscr{L}_h + \mathscr{L}_h(\Pi-\mathrm{Id})$. Using this and Lemma \ref{lem.fstarquasi} we infer (writing $\mathscr{L}, \mathscr{L}_{\ell}, \psi_\ell$ instead of $\mathscr{L}_h,\mathscr{L}_{h,\ell},\psi_{h,\ell}$ to lighten the notations)
\begin{align*}
\langle \mathscr{L} g_\ell, g_\ell \rangle & = \langle \mathscr{L} \Pi \chi_\ell \psi_{\ell}, \Pi \chi_\ell \psi_{\ell} \rangle \\ &= \langle \mathscr{L} \chi_\ell \psi_{\ell}, \chi_\ell \psi_{\ell} \rangle + \langle \mathscr{L} (\Pi-\mathrm{Id}) \chi_\ell \psi_{\ell}, (\Pi-\mathrm{Id}) \chi_\ell \psi_{\ell} \rangle \\ & = \langle \mathscr{L} \chi_\ell \psi_{\ell}, \chi_\ell \psi_{\ell} \rangle + \widetilde{\mathcal{O}}(e^{-2S/\sqrt{h}}). 
\end{align*} 

Then by the fact that $k_\ell \chi_\ell = 0$ and $(\chi_\ell -1)\psi_{\ell}, \sqrt{\chi_\ell-1} \psi_{\ell} = \widetilde{\mathcal{O}}(e^{-S/\sqrt{h}})$ we deduce 
\begin{equation*}
\langle \mathscr{L} g_\ell , g_\ell \rangle = \langle \mathscr{L}_{\ell} \psi_{\ell} , \chi_\ell \psi_{\ell}  \rangle + \langle (\chi_\ell - 1) \psi_{\ell} , \mathscr{L} \psi_{\ell} \rangle + \langle \mathscr{L} (\chi_\ell-1) \psi_{\ell}, (\chi_\ell -1) \psi_{\ell} \rangle + \widetilde{\mathcal{O}}(e^{-2S/\sqrt{h}}) 
\end{equation*}
We have  $\langle \mathscr{L} \psi_{\ell} , \chi_\ell \psi_{\ell} \rangle = \langle \mathscr{L}_{\ell} \psi_{\ell} , \chi_\ell \psi_{\ell} \rangle = \mu(h)\langle \psi_{\ell}, \chi_\ell \psi_{\ell} \rangle = \mu(h) + \widetilde{\mathcal{O}}(e^{-2S/\sqrt{h}})$ and
$\langle (\chi_\ell -1) \psi_{\ell}, \mathscr{L} \psi_{\ell} \rangle = \langle (\chi_\ell -1)\psi_{\ell}, (\mu(h) - k_\ell)\psi_{\ell} \rangle = \widetilde{\mathcal{O}}(e^{-2S/\sqrt{h}})$, $\langle \mathscr{L} (\chi_\ell-1) \psi_{\ell}, (\chi_\ell -1) \psi_{\ell} \rangle = \widetilde{\mathcal{O}}(e^{-2S/\sqrt{h}})$
therefore
$\langle \mathscr{L} g_\ell , g_\ell \rangle = \mu(h) + \widetilde{\mathcal{O}}(e^{-2S/\sqrt{h}}).$ 	
\end{proof}

	Let us consider the new family
	\[\begin{pmatrix}
		\mathfrak{g}_\ell\\
		\mathfrak{g}_r
	\end{pmatrix}=G^{-\frac12}\begin{pmatrix}
		g_\ell\\
		g_r
	\end{pmatrix}\,,\]
	which is orthonormal since
	\[\begin{pmatrix}
		\mathfrak{g}_\ell\\
		\mathfrak{g}_r
	\end{pmatrix}\cdot\begin{pmatrix}
		\mathfrak{g}_\ell&\mathfrak{g}_r
	\end{pmatrix}=G^{-\frac12}\begin{pmatrix}
		g_\ell\\
		g_r
	\end{pmatrix}\cdot\begin{pmatrix}
		g_\ell&g_r
	\end{pmatrix}G^{-\frac12}=\mathrm{Id}\,.\]
	The matrix of $\mathscr{L}_h$ in the orthonormal basis $(\mathfrak{g}_\ell,\mathfrak{g}_r)$ is $G^{-\frac12}LG^{-\frac12}$ and we have
	\[\begin{split}
		G^{-\frac12}LG^{-\frac12}&=\left(1-\frac{\mathrm{T}}{2}\right)\left(\begin{pmatrix}
			\mu(h)& w_h\\
			\overline{w}_h&\mu(h)
		\end{pmatrix}+\mu(h)\mathrm{T}\right)\left(1-\frac{\mathrm{T}}{2}\right)+\widetilde{\mathcal{O}}(e^{-2S/\sqrt{h}})\\
		&=\begin{pmatrix}
			\mu(h)& w_h\\
			\overline{w}_h&\mu(h)
		\end{pmatrix}+\widetilde{\mathcal{O}}(e^{-2S/\sqrt{h}})\,.
	\end{split}\]
The splitting of $\begin{pmatrix}
			\mu(h)& w_h\\
			\overline{w}_h&\mu(h)
		\end{pmatrix}$ is $2 |w_h|$ thus by the $\min$-$\max$ theorem for hermitian matrix this proves Proposition \ref{prop.splitformula}.

\subsection{Estimate of the interaction}
Theorem \ref{thm.main} is a consequence of Proposition \ref{prop.splitformula} and the following proposition.

\begin{proposition}\label{prop.wh}
We have 
\[w_h=(1+o_{h\to 0}(1))  2 \left( \frac{a''(0)}{2} \right)^{\frac14} h^{\frac54}\sqrt{V(0)}\left(\frac{\kappa}{\pi}\right)^\frac12 \exp\left(-\int_{x_\ell}^{0} \frac{\partial_s\sqrt{V(s)}-\kappa}{\sqrt{V(s)}}\mathrm{d}s \right)e^{-\frac{S}{\sqrt{h}}}\,.
\]
\end{proposition}

\begin{proof}
From Proposition \ref{prop.splitformula} it suffices to compute $w_h$. We start with the following lemma.
\begin{lemma}
We have \[
\begin{split}
w_h := \langle (\mathscr{L}_{h,\ell} - \mu)\chi_\ell \psi_{h,\ell},  \chi_r \psi_{h,r} \rangle = e^{-\frac{S}{\sqrt{h}}}  \langle ( \mathscr{L}_{h,\ell}^{\Phi_\ell}-\mu)e^{\frac{\Phi_\ell}{\sqrt{h}}}\chi_\ell \psi_{h,\ell}, e^{\frac{\Phi_r}{\sqrt{h}}} \chi_r \psi_{h,r} \rangle \\ + \mathcal{O}(h^\infty e^{-\frac{S}{\sqrt{h}}}).
\end{split} \]
\end{lemma}
\begin{proof}
The main obstacle comes from the fact that $\Phi_\ell(x) + \Phi_r(x) = S$ is not satisfied everywhere between $x_\ell$ and $x_r$.

Let us define 
\[\Psi_\ell(x) = \mathrm{sgn}(x-x_\ell)\int_{x_\ell}^x \sqrt{b(s,0)} \mathrm{d} s \]
and notice that for $x \in \mathrm{supp} \chi_r$, $\Psi_\ell(x) + \Phi_r(x) = \int_{x_\ell}^{x_r} \sqrt{b(s,0)}\mathrm{d}s = S$. We write
\begin{align*}
(\mathscr{L}_{h,\ell} - \mu)  \chi_\ell \psi_{h,\ell} = e^{-\frac{\Psi_\ell}{\sqrt{h}}} e^{\frac{\Psi_\ell- \Phi_\ell}{\sqrt{h}}} ( \mathscr{L}_{h,\ell}^{\Phi_\ell}-\mu)e^{\frac{\Phi_\ell}{\sqrt{h}}}\chi_\ell \psi_{h,\ell}. 
\end{align*} 
and estimate
\begin{multline}
|w_h -e^{-\frac{S}{\sqrt{h}}}  \langle ( \mathscr{L}_{h,\ell}^{\Phi_\ell}-\mu)e^{\frac{\Phi_\ell}{\sqrt{h}}}\chi_\ell \psi_{h,\ell}, e^{\frac{\Phi_r}{\sqrt{h}}} \chi_r \psi_{h,r} \rangle | \\= |e^{-\frac{S}{\sqrt{h}}}  \langle \bigl(e^{\frac{\Psi_\ell- \Phi_\ell}{\sqrt{h}}} - 1 \bigr) ( \mathscr{L}_{h,\ell}^{\Phi_\ell}-\mu)e^{\frac{\Phi_\ell}{\sqrt{h}}}\chi_\ell \psi_{h,\ell}, e^{\frac{\Phi_r}{\sqrt{h}}} \chi_r \psi_{h,r} \rangle|,
\end{multline}
using $\underline{\chi}_\ell \in \mathscr{C}_0^{\infty}(\R)$ satisfying $\underline{\chi}_\ell \chi_\ell = \chi_\ell$ (in particular $\underline{\chi}_\ell' \chi_\ell = 0$) and $\underline{\chi}_\ell k_\ell = 0$. Therefore we obtain by the pseudodifferential calculus
\[ \| [\mathscr{L}_{h,\ell}^{\Phi_\ell}, \underline{\chi_\ell}]\chi_\ell  e^{\frac{\Phi_\ell}{\sqrt{h}}}\psi_{h,\ell} \| = \mathcal{O}(h^\infty) \| \chi_\ell  e^{\frac{\Phi_\ell}{\sqrt{h}}}\psi_{h,\ell} \|. \]
Then also using $|\chi_r \bigl(e^{\frac{\Psi_\ell- \Phi_\ell}{\sqrt{h}}} - 1 \bigr)| \leq 1$ et $\chi_r \underline{\chi}_\ell \bigl(e^{\frac{\Psi_\ell- \Phi_\ell}{\sqrt{h}}} - 1 \bigr) = 0$ we deduce,
\begin{multline}
\label{eq:raym_remainder}
|w_h -e^{-\frac{S}{\sqrt{h}}} \langle ( \mathscr{L}_{h,\ell}^{\Phi_\ell}-\mu)e^{\frac{\Phi_\ell}{\sqrt{h}}}\chi_\ell \psi_{h,\ell}, e^{\frac{\Phi_r}{\sqrt{h}}} \chi_r \psi_{h,r} \rangle | \\ = |e^{-\frac{S}{\sqrt{h}}}  \langle \bigl(e^{\frac{\Phi_\ell-\Psi_\ell}{\sqrt{h}}} - 1 \bigr) \underline{\chi}_\ell ( \mathscr{L}_{h,\ell}^{\Phi_\ell}-\mu)e^{\frac{\Phi_\ell}{\sqrt{h}}}\chi_\ell \psi_{h,\ell}, e^{\frac{\Phi_r}{\sqrt{h}}} \chi_r \psi_{h,r} \rangle| \\   + \mathcal{O}(h^\infty e^{-\frac{S}{\sqrt{h}}}) \| \|e^{\frac{\Phi_\ell}{\sqrt{h}}}\chi_\ell \psi_{h,\ell} \| \| e^{\frac{\Phi_r}{\sqrt{h}}} \chi_r \psi_{h,r} \| \\
= \mathcal{O}(h^\infty e^{-\frac{S}{\sqrt{h}}}) \|e^{\frac{\Phi_\ell}{\sqrt{h}}}\chi_\ell \psi_{h,\ell} \| \| e^{\frac{\Phi_r}{\sqrt{h}}} \chi_r \psi_{h,r} \|.
\end{multline}
Note then that, for instance thanks to Corollary \ref{cor.wkb}, $\|e^{\frac{\Phi_\ell}{\sqrt{h}}}\chi_\ell \psi_{h,\ell} \|$ and $\| e^{\frac{\Phi_r}{\sqrt{h}}} \chi_r \psi_{h,r} \|$ are of order $\mathcal{O}(h^r)$ for some $r \in \mathbb{Q}$. We infer that the remainder in \eqref{eq:raym_remainder} is $\mathcal{O}(h^\infty e^{-\frac{S}{\sqrt{h}}})$.
\end{proof}

 Support considerations  and Corollary \ref{cor.wkb} give 
\[\| \chi_\star e^{\frac{\Phi_\star}{\sqrt{h}}}(c(h)\psi_{h,\star} - \Psi_\star^{\mathrm{WKB}}) \| = \mathcal{O}(h^\infty)\,,\]
where $c(h)$ is defined in \eqref{eq.ch}. Hence,
\begin{equation}\label{eq.whapp}
	\begin{split}
\langle & (\mathscr{L}_h - \mu(h))\chi_\ell \psi_{h,\ell},\chi_r\psi_{h,r} \rangle \\ & =  |c(h)|^{-2}e^{-\frac{S}{\sqrt{h}}}\langle (\mathscr{L}_{h,\ell}^{\Phi_\ell} - \mu(h))\chi_\ell e^{\frac{\Phi_\ell}{\sqrt{h}}}\Psi^{\mathrm{WKB}}_{\ell}, e^{\frac{\Phi_r}{\sqrt{h}}} \chi_r \Psi_{r}^{\mathrm{WKB}} \rangle + \mathcal{O}(h^{\infty}e^{-\frac{S}{\sqrt{h}}}) \\
 &= (1+o(1)) h^{-\frac14} e^{-\frac{S}{\sqrt{h}}}\langle (\mathscr{L}_{h,\ell}^{\Phi_\ell} - \mu(h)) \chi_{\ell} u_{\ell,h}, \chi_r u_{r,h} \rangle+ \mathcal{O}(h^{\infty}e^{-\frac{S}{\sqrt{h}}})\,.
\end{split}
\end{equation}
 Thanks to Lemma \ref{lem:apxl} (or recalling the proof of Proposition \ref{prop.WKB}),
\begin{equation*} 
	\langle (\mathscr{L}_{h,\ell}^{\Phi_\ell} - \mu(h)) \chi_{\ell} u_{\ell,h}, \chi_r u_{r,h} \rangle=   h^{\frac32}\langle (P_3(x,D_x)-c_0)( \chi_\ell u_\ell), \chi_r u_r \rangle  + o(h^{\frac32}) . 
\end{equation*}
Since, by construction, $(P_3(x,D_x)-c_0) u_\ell = 0$  (see \eqref{WKBeq}), we get
\begin{equation*}
	\langle (\mathscr{L}_{h,\ell}^{\Phi_\ell} - \mu(h)) \chi_{\ell} u_{\ell,h}, \chi_r u_{r,h} \rangle =  h^{\frac32}\langle a''(0)\Phi_\ell' \chi_\ell' u_\ell,\chi_r u_r \rangle  +  o(h^\frac32 )\,.
\end{equation*}
Noticing that $\chi_\ell ' \chi_r = \mathds{1}_{[x_r-2\eta,x_r-\eta]} \cdot \chi_\ell'$, we get
\[ 	\langle (\mathscr{L}_{h,\ell}^{\Phi_\ell} - \mu(h)) \chi_{\ell} u_{\ell,h}, \chi_r u_{r,h} \rangle= a''(0) h^\frac32  \int_{x_r - 2\eta}^{x_r - \eta} \chi_\ell'(x) \Phi_\ell'(x)u_\ell(x)\overline{u}_r(x)\mathrm{d}x + o(h^\frac32)\,. \]
By definition of $\chi_\ell$, we have $\int_{x_r - 2\eta}^{x_r - \eta} \chi_\ell'(x)\mathrm{d}x = -1$. In fact, the function $\Phi_\ell' u_\ell \overline{u}_r$  is constant and equal to $\Phi_\ell'(0) |u_\ell(0)|^2$. Indeed, by \eqref{WKBeq} $b_\ell(x,\xi) = b(x,\xi) + k_\ell(x)$ we obtain $\partial_\xi b_\ell = \partial_\xi b$ and
\begin{equation}
\begin{split}
\left(i \frac{\Phi_\ell'(x) \partial_\xi b(x,0)}{a''(0)} + \frac{1}{2}\Phi_\ell''(x) +  \Phi_\ell'(x) \partial_x-\lambda_3 \right)u_\ell =0\,, \\
\left(i \frac{\Phi_r'(x) \partial_\xi b(x,0)}{a''(0)} + \frac{1}{2}\Phi_r''(x) + \Phi_r'(x) \partial_x-\lambda_3 \right)u_r =0\,.
\end{split}
\end{equation}
 Using then the fact that $\Phi_r + \Phi_\ell$ is constant on $(x_\ell + \eta, x_r - \eta)$ gives, denoting $v(x) = \frac{\Phi_\ell'(x) \partial_\xi b(x,0)}{a''(0)}$,
	\begin{equation*}
	\Phi_\ell'(x) \partial_x u_\ell = (-iv(x) - \frac{1}{2}\Phi_\ell''(x) + \lambda_3) u_\ell \hbox{ and } \Phi_\ell'(x) \partial_x \overline{u}_r = (iv(x) - \frac{1}{2}\Phi_\ell''(x)-\lambda_3) \overline{u}_r\,.
\end{equation*}
which implies that $\partial_x \left(\Phi_\ell' u_\ell \overline{u}_r \right) = 0$. There remains to use the explicit formula \eqref{eq.princip} and to recall \eqref{eq.whapp} to end the proof.

\end{proof}

\appendix

\section{Pseudo-differential tools}\label{sec.A}
\label{apx.1}

\subsection{Notation}
Let us recall some usual notation (see, for instance, \cite[Chapter 4]{Zworski}).

For $\delta_1, \delta_2 \in (0,1)$, we consider
\begin{equation}\label{eq.Sd1d2}
\begin{split}
	S_{\delta_1,\delta_2} &(\R^2)  \\ & := \{ q_h \in \mathscr{C}^{\infty}(\R^2) ~ ; ~ \forall (\gamma_1,\gamma_2) \in \N^2, ~ \exists C_\gamma > 0, ~~ | \partial_x^{\gamma_1} \partial_\xi^{\gamma_2} q_h | \leq C_\gamma h^{-\delta_1 \gamma_1} h^{-\delta_2 \gamma_2} \}\, .
\end{split}
\end{equation}
We let $S_{\delta}(\R^2) := S_{\delta,\delta}(\R^2)$ and $S(\R^2)=S_0(\R^2)$.

\begin{enumerate}[---]
\item We say that a symbol $q_h \in S_{\delta}(\R^2)$ has asymptotic expansion $q_0 + h^{\frac12}q_1+hq_2 + \cdots$ with $q_j \in S_{\delta}(\R^2)$ for each $j \in \N$, if
\begin{equation}
\forall \gamma \in \N^2, ~ \forall m \in \N,~ \exists C_{\gamma,m}\in \R_+, ~~ \left|\partial^{\gamma}\left(q_h - \sum_{j=0}^{m-1} h^{\frac{j}{2}} q_j\right)\right| \leq C_{\gamma,m} h^{\frac{m}{2}} h^{-|\gamma|\delta}.
\end{equation} 
\item We say that $(u_h)_{h \in (0,h_0)}$ a family of elements of\footnote{ $S_\delta(\R) = \{ u_h \in \mathscr{C}^{\infty}(\R) ~ | ~ \forall \gamma \in \N, ~ \exists C_\gamma >0 ~ |\partial^{\gamma} u_h| \leq C_\gamma h^{-\delta \gamma} \}$} $S_{\delta}(\R)$, has asymptotic expansion $u_0 + h^\frac12 u_1 + h u_2 + h^\frac32 u_3+\cdots$ with each $u_j \in S_\delta(\R)$ if
\begin{equation}
\forall \gamma \in \N,~ \forall m \in \N, ~ \exists C_{\gamma,m}\in \R_+, ~~ \left|\partial^{\gamma}\left(u_h - \sum_{j=0}^{m-1} h^{\frac{j}{2}}u_j\right)\right| \leq C_{\gamma,m} h^{\frac{m}{2}} h^{-\gamma \delta}.
\end{equation}
\item We say that $(\lambda_h)_{h \in (0,h_0)} \in \R^{(0,h_0)}$ has asymptotic expansion $\lambda_0 + h^\frac12 \lambda_1 + h \lambda_2 + h^\frac32 \lambda_3+\cdots$ if
\begin{equation}
\forall m \in \N, ~ \exists C_{m}\in \R_+, ~~ \left|\lambda_h - \sum_{j=0}^{m-1} h^{\frac{j}{2}} \lambda_j\right| \leq C_{m} h^{\frac{m}{2}}.
\end{equation}
\end{enumerate}

\subsection{Gärding, Fefferman-Phong inequalities and the Kuranishi trick}
\label{apx.2}
We recall some results of \cite{bony} concerning the Gärding and Fefferman-Phong inequality. The first estimate concerns the class of symbols $S(\R^2)$. Letting $q_h \in S(\R^2)$ non-negative, by the Gärding inequality, there exists $C>0$ such that 
\begin{equation}
\label{eq:semgarding}
\forall u \in L^2(\R)\,,\quad \langle \mathrm{Op}_h^w(q_h) u, u\rangle \geq -Ch\| u \|^2.
\end{equation}
This result can be refined thanks to the Fefferman-Phong inequality
\begin{equation}
	\forall u \in L^2(\R)\,,\quad \langle \mathrm{Op}_h^w(q_h) u, u\rangle \geq -Ch^2\| u \|^2.
\end{equation}

These estimates can be extended to the exotic class of symbols.

\begin{theorem}[Fefferman-Phong inequality]
	\label{th:fph}
	Let $\delta_1, \delta_2 \in (0,1)$ such that $\delta_1 + \delta_2 < 1$. If $q_h \in S_{\delta_1,\delta_2}(\R)$ is non-negative, then
	\begin{equation}
		\forall u \in L^2(\R)\,\quad \langle q_h^w u,u \rangle \geq -C h^{2-2(\delta_1 + \delta_2)} \| u \|^2\,.
	\end{equation}
\end{theorem}
\begin{proof}
	Let $\alpha \in \R$ and consider the isometric scaling:
	\begin{equation}
		V_{\alpha} : \left \{ \begin{array}{ccc}
			L^2(\R) & \longrightarrow & L^2(\R)\\
			u & \longmapsto & h^{-\alpha/2}u(h^{-\alpha}\cdot)
		\end{array} \right.\,.
	\end{equation}
	Notice that $V_{\alpha}^{-1} = V_{\alpha}^* = V_{-\alpha}$ thus
	\begin{equation}
		V_{-\delta_1} \mathrm{Op}_h^w(q_h(x,\xi)) V_{\delta_1} = \mathrm{Op}_h^w(q_h(h^{\delta_1}x,h^{-\delta_1}\xi)) = \mathrm{Op}_{h^{1-(\delta_1 + \delta_2)}}(q_h(h^{\delta_1}x,h^{\delta_2}\xi)),
	\end{equation}
	with $q_h(h^{\delta_1}\cdot,h^{\delta_2}\cdot) \in S(\R^2)=S_{0,0}(\R^2)$. The Fefferman-Phong inequality in $S(\R^2)$ (with the new semiclassical parameter $h^{1-\delta_1-\delta_2}$) yields the result.
\end{proof}

In the following Lemma, the function $\varphi$ satisfies the assumptions made at the beginning of Section \ref{section:exp}. Its proof is an adaptation of \cite[Section 3]{Shu1998}.
\begin{lemma}[First Kuranishi trick]
	\label{lem:kur1}
	The operator $\mathscr{L}_{h,\ell}^{\varphi}$ defined in \eqref{eq.Lphi} is a pseudo-differential operator of symbol \begin{equation}
		\label{eq:pphi}
		q_{\varphi}(x,\xi) = a(\xi+ih^\frac12 \varphi') + hb_{\ell}(x,\xi + ih^\frac12 \varphi') + \mathcal{O}_{S_{1/4}(\R)}(h^{3/2+1/4}).
	\end{equation}
	
\end{lemma} 
\begin{proof}
For all $u \in \mathscr{S}(\R)$, denoting by $p(x,\xi)$ the symbol of $\mathscr{L}_{h,\ell}$,
	
	\begin{align*}
		\mathscr{L}_{h,\ell}^{\varphi}u & := e^{\frac{\varphi}{h^{1/2}}}\mathscr{L}_{h,\ell}e^{-\frac{\varphi}{h^{1/2}}}u = \tfrac{1}{2\pi h}\int_{\R} \int_{\R} p\left(\tfrac{x+y}{2},\xi \right)e^{\frac{i}{h}\left[(x-y)\xi-ih^{1/2}(\varphi(x)-\varphi(y)) \right]} u(y)\mathrm{d}y \mathrm{d}\xi \\
		&= \tfrac{1}{2\pi h} \int_{\R} \int_{\R-ih^{1/2} \theta(x,y)} p\left(\tfrac{x+y}{2},\xi + ih^{1/2}\theta(x,y) \right)e^{\frac{i}{h}(x-y)\xi}u(y) \mathrm{d}y \mathrm{d}\xi\,,
	\end{align*}
with $\theta :
		\begin{array}{ccc}
			(x,y) & \longmapsto & \left \{  \begin{array}{ll} 
				\frac{\varphi(x) - \varphi(y)}{x-y} \hbox{ if } x \neq y, \\
				\varphi'(x) \quad ~ \hbox{ if } x= y. 
			\end{array}\right. 
		\end{array}$. Note that
	$\theta(x,y) = \int_0^1 \varphi'(y+t(x-y))\mathrm{d}t,$ which proves that $\theta \in S_{\frac14}(\R^2)$ and is bounded by $\lVert \varphi' \rVert_{L^{\infty}}$. Thanks to Cauchy's theorem and Riemann-Lebesgue's lemma, we have, for all $u \in \mathscr{S}(\R)$,
	\begin{align*}
		& \int_{\R} \int_{[-R,R]} p\left(\tfrac{x+y}{2},\xi + ih^{1/2}\theta(x,y) \right)e^{\frac{i}{h}(x-y)\xi}u(y)\mathrm{d}\xi \mathrm{d}y \\ 
		& - \int_{\R} \int_{[-R,R]-ih^{1/2} \theta(x,y)} p\left(\tfrac{x+y}{2},\xi + ih^{1/2}\theta(x,y) \right)e^{\frac{i}{h}(x-y)\xi}u(y)\mathrm{d}\xi \mathrm{d}y \xrightarrow[R \to +\infty]{} 0\,.
	\end{align*}
	This eventually proves that, in the sense of oscillatory integrals, we have
	\begin{equation}
		\mathscr{L}_{h,\ell}^\varphi u = \frac{1}{2\pi h} \iint_{\R^2} p\left( \tfrac{x+y}{2}, \xi + ih^\frac12 \theta(x,y) \right) e^{\frac{i}{h}(x-y)\xi}u(y)\mathrm{d}y\mathrm{d}\xi\,.
	\end{equation}
This equality extends to $L^2(\mathbb{R})$ by continuity (by the standard theory). Let us now use this formula to prove that  $\mathscr{L}_{h,\ell}^\varphi$  is a pseudo-differential operator. To do so, let us use the so-called Kuranishi trick.  The Taylor formula gives
	\begin{multline}
	\label{eq:classeI}
		\varphi(x) = \varphi \left( \frac{x+y}{2} \right) + \left( \frac{x-y}{2}\right)\varphi'\left( \frac{x+y}{2}\right) + \frac{\left (\frac{x-y}{2} \right)^2}{2!}\varphi^{(2)}\left( \frac{x+y}{2} \right) \\ +\left (\frac{x-y}{2} \right)^3 \int_0^1 \frac{(1-t)^2}{2!} \varphi^{(3)}\left(\frac{x+y}{2}+t \frac{x-y}{2} \right)\mathrm{d}t, 
	\end{multline}
and a similar expression for $\varphi(y)$. Hence, there is a smooth function $I(x,y)$ such that
	\begin{equation}
		\theta(x,y) = \varphi'\left( \frac{x+y}{2}\right) + \left(x-y\right)^2 I(x,y)\,,
	\end{equation}
	where $I(x,y)$ is expressed in terms of $\varphi^{(3)}$ only thus $I(x,y)$ belongs to $h^{-1/4}S_{\frac14,\frac14}(\R)$. Since the map $\R^2 \ni (x,y) \longmapsto \frac{\varphi(x)-\varphi(y)}{x-y}$ is bounded and $\varphi'$ as well, we obtain that the map $\R^2 \ni (x,y) \longmapsto (x-y)^2I(x,y)$ is bounded. Therefore by the Taylor expansion at order $1$ in $h$, 
	
	\begin{align*}
	p\left( \frac{x+y}{2},\xi + i h^\frac12 \theta(x,y) \right) & = p\left( \frac{x+y}{2},\xi + i h^\frac12 \varphi'\left( \frac{x+y}{2} \right) \right) + h^{\frac12}(x-y)^2 r_h(x,y,\xi)
\end{align*}
where $r_h(x,y,\xi) \in h^{-\frac14} S_{\frac14,\frac14,0}(\R^3)$ (this space being the natural adaptation to $\R^3$ of Definition \ref{eq.Sd1d2}). Explicitely,
 \[ r_h(x,y,\xi) = i   I(x,y) \int_0^1 \partial_\xi p(\frac{x+y}{2}, \xi + ih^\frac12 \varphi'(\frac{x+y}{2}) + it h^\frac12 (x-y)^2 I(x,y))\mathrm{d}t\,. \]
  For all $u \in \mathscr{S}(\R)$, by integration by parts,
 \begin{align*}
 R_hu&=	\frac{1}{2\pi h}\iint_{\R^2} e^{\frac{i}{h}(x-y)\xi}(x-y)^2r_h(x,y,\xi)u(y)\mathrm{d}\xi \mathrm{d}y \\
 & = \frac{1}{2\pi h}\iint_{\R^2} e^{\frac{i}{h}(x-y)\xi}(-hD_\xi)^2 r_h(x,y,\xi)u(y)\mathrm{d}\xi \mathrm{d}y\,,
 \end{align*}
where we used that $(x-y) e^{\frac{i}{h}(x-y)\xi} = hD_\xi e^{\frac{i}{h}(x-y)\xi}$. Thanks to the transformation formula \cite[Theorems 4.20 \& 4.21]{Zworski}, there exists $\widetilde{r}_h \in h^{-\frac14} S_\frac14(\R^2) $ such that  $R_h= h^2 \mathrm{Op}_h^w(\widetilde{r}_h)$. This yields $q_{\varphi} \in S_{\frac14}(\R^2)$ such that 
	\begin{equation}
		\mathscr{L}_{h,\ell}^{\varphi} = \mathrm{Op}_h^w(q_{\varphi}(x,\xi)) \hbox{ with } q_{\varphi}(x,\xi) = p(x,\xi + i h^\frac12 \varphi') +   \mathcal{O}_{S_{\frac14}(\R^2)}(h^{3/2 +1/4}).
	\end{equation}
\end{proof}

\begin{lemma}[Second Kuranishi trick]\label{lem.A3}
	\label{lem:kur2}
	The operator $\mathscr{L}_{h,\ell}^{\Phi_\ell}$ defined in \eqref{eq.Lphi} is a pseudo-differential operator of symbol \begin{equation}
		\label{eq:pphi2}
		q_{\Phi_\ell}(x,\xi) = a(\xi+ih^\frac12 \Phi_\ell') + hb_{\ell}(x,\xi + ih^\frac12 \Phi_\ell') + h^2 r_h
	\end{equation}
where $r_h \in S(\R^2)$ has a full asymptotic expansion in $S(\R^2)$, $r_h = r_0 + h^{\frac12} r_1 + hr_2 + \cdots$ with $r_0,r_1, \cdots$ independent of $h$.
\end{lemma} 
The proof of this lemma is an adaptation of the proof of the previous lemma and left to the reader.

\section{Stationary phase and WKB constructions}\label{sec.B}
\begin{lemma}
\label{lem:phistat}
Let $p_h \in S(\R^2)$ having asymptotic expansion $p_0 + h^\frac12 p_1 + \cdots $ and $u_h \in \mathscr{C}_0^{\infty}(\R)$ (independent of $h$) having asymptotic expansion $u_0 + h^\frac12 u_1 + \cdots $ with each $u_k \in \mathscr{C}_0^{\infty}(\R)$. The function $\mathrm{Op}_h^w(p_h) u_h$ belongs to $\mathscr{C}^{\infty}(\R)$ and has an asymptotic expansion in powers of $h$ in $S(\mathbb{R})$ given by:
\begin{equation}
\mathrm{Op}_h^w(p_h) u_h \sim \sum_{k \in \N} \frac{(hD_y)^k}{k!} \left(\partial_{\xi}^k p_h(x+y/2,0) u_h(x+y) \right)|_{y=0}\,.
\end{equation}
\end{lemma}

\begin{proof}
It suffices to prove the lemma when $p_h(x,\xi) = p(x,\xi)$ and $u_h = u$ are independent of $h$. Note that
\begin{align*}
\mathrm{Op}_h^w(p)u & = \frac{1}{2\pi h}\iint_{\R^{2}} e^{i(x-y)\xi/h} p\left( \frac{x+y}{2},\xi\right) u(y)\mathrm{d}y \mathrm{d}\xi \\ & = \frac{1}{2\pi h} \iint_{\R^{2}} e^{-iy\xi/h} p\left( \frac{2x+y}{2},\xi\right) u(x+y)\mathrm{d}y \mathrm{d}\xi\,.
\end{align*}
Thanks to the stationary phase results \cite[Theorems 3.17 \& 4.17]{Zworski}, we get
\begin{equation}
\mathrm{Op}_h^w(p)u \sim \left. \sum_{k \in \N} \frac{(hD_y)^k}{k!} \left(\partial_{\xi}^k p(x+y/2,0)u(y) \right) \right|_{y=0}.
\end{equation}
\end{proof}

\begin{lemma}
\label{lem:apxl}
There exist differential operators $(P_{\gamma}(x,D_x))_{0 \leq \gamma \leq J}$ such that,
\[\mathscr{L}_{h,\ell}^{\Phi_\ell}(\chi u_h^{[J]}) = \sum_{0 \leq \gamma \leq J+3} h^{\gamma/2}P_{\gamma}(x,D_x) (\chi u_h^{[J]}) + \mathcal{O}_{L^\infty(\R)}(h^{\frac{J+4}{2}}) \,. \]
The first differential operators are
\begin{equation}
\label{eq:opdiff2}
\left \{ \begin{array}{lll}
P_0(x,D_x) = a(0) = 0, \\
P_1(x,D_x) = i \Phi_\ell'(x) a'(0) = 0, \\
P_2(x,D_x) = -\frac{a''(0)}{2} \Phi_\ell'(x)^2 + b_\ell(x,0) = 0, \\
 P_3(x,D_x)= i \Phi_\ell'(x)\partial_\xi b_\ell(x,0) + \frac{a''(0)}{2} \Phi_\ell''(x) +  a''(0)\Phi_\ell'(x) \partial_x\,.
\end{array} \right.
\end{equation}
\end{lemma} 
\begin{proof}
Thanks to the holomorphy assumptions on $a$ and $b$, we can use Lemma \ref{lem.A3}. Then, applying Lemma \ref{lem:phistat}, we get
\begin{equation*}
\begin{aligned}
 \mathscr{L}^{\Phi_{\ell}}_{h,\ell} (\chi u_h^{[J]}) & =  \sum_{j=0}^{\frac{J+3}{2}} \mathrm{Op}_h^w\left ( \frac{\xi^j}{j!} \partial_{\xi}^j q_{\Phi_\ell}(x,0) \right) (\chi u_h^{[J]}) + \mathcal{O}_{L^\infty(\R)}(h^{\frac{J+4}{2}}) \\
& =  \sum_{\gamma=0}^{J+3} h^{\gamma/2}P_{\gamma}(x,D_x)(\chi u_h^{[J]})+ \mathcal{O}_{L^\infty(\R)}(h^{\frac{J+4}{2}})\,.
\end{aligned}
\end{equation*}
The operators $P_{\gamma}(x,D_x)$ are differential operators of degree lower than $\gamma/2$ obtained using the expansion of $q_{\Phi_\ell}(x,0)$ in powers of $h^{\frac12}$. Explicitely, for the first few terms ($J=0$), we can write the Taylor expansion of the symbol in $\xi + ih^\frac12 \Phi_\ell'$  and we get
\begin{equation*} 
\begin{split}
\mathscr{L}^{\Phi_\ell}_{h,\ell}(\chi u_h) = \left( \frac{a''(0)}{2}(h D_x + i h^{\frac12} \Phi_\ell'(x))^2 +hb_{\ell}(x,0) + ih^{\frac32} \partial_\xi b_{\ell}(x,0) \Phi_\ell'(x) \right) \chi u_h  \\ + \mathcal{O}_{L^\infty(\R)}(h^2)\,.
\end{split}
\end{equation*} 
\end{proof}

The following last lemma proves that the norm of our normalized WKB quasimodes is not $\mathcal{O}(h^\infty)$. We consider a symbol $u_h$ with asymptotic expansion $u_0 + h^\frac12 u_1 + \cdots$, a cutoff $\chi \in \mathscr{C}_0^{\infty}(\R)$ such that $\chi= 1$ near $x_\ell$.
\begin{lemma}[WKB and Laplace Integral]
\label{lem:lap}
Let us suppose that there exists $m \in \N$ such that $u_0^{(m)}(x_\ell) \neq 0$ then $\| \chi u_h e^{-\Phi_\ell/\sqrt{h}} \|_{L^2(\R)} \geq c h^{\frac18+\frac{m}{4}}$.
\end{lemma}
\begin{proof}
Near $x_\ell$ the application $\phi_\ell := \mathrm{sgn}(x-x_\ell)\sqrt{\Phi_\ell}$ is a $\mathscr{C}^{\infty}$ diffeomorphism satisfying $\phi_\ell^2 = \Phi_\ell$ and we have the approximation $\phi_\ell(x) = \sqrt{\frac{\Phi_\ell''(x_\ell)}{2}}(x-x_\ell)+o(x-x_\ell)$. Therefore, modulo a term of order $\mathcal{O}(e^{-\frac{\epsilon}{\sqrt{h}}})$ for some $\epsilon > 0$ we can write
\begin{equation}
\label{integ}
\begin{aligned}
\lVert  \chi u_h & e^{-\Phi_\ell/\sqrt{h}} \rVert^2  =  \int_{\mathrm{Neigh}(0)}e^{-\frac{2y^2}{\sqrt{h}}} |u_h\circ \phi_\ell^{-1}(y)|^2 (\phi_\ell^{-1})'(y)\mathrm{d}y+\mathcal{O}(e^{-\frac{\epsilon}{\sqrt{h}}}) \\
& = \frac{h^{\frac14}}{\sqrt{2}} \int_{h^{-\frac14}\mathrm{Neigh}(0)}e^{-y^2}|u_h\circ \phi_\ell^{-1}(h^{\frac14}y/\sqrt{2})|^2 (\phi_\ell^{-1})'(h^{\frac14}y/\sqrt{2})\mathrm{d}y+\mathcal{O}(e^{-\frac{\epsilon}{\sqrt{h}}}) \,.
\end{aligned}
\end{equation}
A Taylor expansion gives for some $r \in [ \![ 0,m]\!]$,
\begin{equation}
u_h\circ \phi_\ell^{-1}\left(\frac{h^{\frac14}y}{\sqrt{2}}\right) = h^{\frac{r}{4}} \left( \sum_{j=0}^{N-1} h^{\frac{j}{4}} P_j(y) + \mathcal{O}(h^{\frac{N}{4}})\right)\,,
\end{equation}
where each $P_j$ is a polynomial of degree at most $j+r$ and $P_0$ is non-zero because of the assumption $u_0^{(m)}(x_\ell) \neq 0$. Therefore, we find other polynomials $Q_j$ such that:
\begin{equation}
|u_h\circ \phi_\ell^{-1}(h^{\frac14}y/\sqrt{2})|^2 (\phi_\ell^{-1})'(h^{\frac14}y/\sqrt{2})=  h^{\frac{r}{2}} \left( \sum_{j=0}^{N-1} h^\frac{j}{4} Q_j + \mathcal{O}(h^{\frac{N}{4}}) \right)\,,
\end{equation}
with $Q_0 = \left( \frac{\Phi_\ell''(x_\ell)}{2} \right)^{-\frac12}P_0^2$.
This yields 
\begin{align*}
\lVert  \chi u_h e^{-\Phi_\ell/\sqrt{h}} \rVert^2 = h^{\frac14} h^{\frac{r}{2}}\left(\Phi_\ell''(x_\ell) \right)^{-\frac12} \int_{h^{-\frac14}\mathrm{Neigh}(0)} e^{-y^2} P_0(y)^2 \mathrm{d}y +\mathcal{O}(h^{\frac{r+1}{2}})\,,
\end{align*}
which concludes the proof.
\end{proof}
 
\section*{Acknowledgments}
This work was conducted within the France 2030 framework programme, Centre Henri Lebesgue ANR-11-LABX-0020-01. Antide Duraffour acknowledges funding by the AMX grant of \'Ecole polytechnique. The first author is grateful to San V\~u Ng\d{o}c, his main thesis supervisor, for his many relevant advices. The authors are also indepted to the thorough suggestions of the anonymous reviewer at MRL. The authors also would like to thank Yannick Guedes-Bonthonneau, Bernard Helffer, and Frédéric Hérau for their insightful comments and for communicating relevant references.
 
\bibliographystyle{abbrv}
\bibliography{biblio}
\end{document}